\documentclass[ejs]{imsart}

\RequirePackage[OT1]{fontenc}
\RequirePackage{amsthm,amsmath,comment,latexsym,amssymb}
\RequirePackage[numbers]{natbib}
\RequirePackage[colorlinks,citecolor=blue,urlcolor=blue]{hyperref}

\RequirePackage{graphicx}

\pubyear{2017}
\volume{11}
\issue{1935-7524}
\firstpage{916}

\startlocaldefs
\numberwithin{equation}{section}
\theoremstyle{plain}
\newtheorem{theorem}{Theorem}[section]
\newtheorem{proposition}{Proposition}[section]
\newtheorem{remark}{Remark}[section]
\newtheorem{lemma}{Lemma}[section]
\newtheorem{corollary}{Corollary}[section]

\newcommand{\R}{\mathbb{R}}

\newcommand{\1}{\mathbb{I}}

\renewcommand{\H}{{\cal{H}}}

\newcommand{\finv}{q}

\newcommand{\K}{{\mathcal KL}}
\newcommand{\KL}{{\mathcal KL}}

\endlocaldefs

\begin{document}

\begin{frontmatter}
\title{Adaptive density estimation based \\ on a mixture of Gammas} 
\runtitle{Density estimation based on a mixture of Gammas}


\author{\fnms{Natalia} \snm{Bochkina}\ead[label=e1]{N.Bochkina@ed.ac.uk}}
\address{University of Edinburgh and Maxwell Institute, United Kingdom \\ \printead{e1}}
\and
\author{\fnms{Judith} \snm{Rousseau}\ead[label=e2]{Rousseau@ceremade.dauphine.fr}}
\address{CEREMADE (Universit\'e Paris Dauphine) and CREST (INSEE, Paris), France \\ \printead{e2}}


\runauthor{N.Bochkina and J.Rousseau}


\begin{abstract}
We consider the problem of Bayesian density estimation on the positive semiline for possibly unbounded densities.
We propose a hierarchical Bayesian estimator based on the gamma mixture prior which can be viewed as a location mixture.
We study convergence rates of Bayesian density estimators  based on such mixtures.
We construct approximations of the local H\"older densities, and of their extension to unbounded densities, to be continuous mixtures
of gamma distributions, leading to approximations of such densities by finite mixtures. These results are then used to derive
posterior concentration rates, with priors based on these mixture models.
The rates are minimax (up to a log n term) and since the priors are independent of the smoothness, the rates are adaptive to the smoothness.

One of the novel feature of the paper is that these results hold for densities with polynomial tails. Similar results are obtained using a hierarchical Bayesian model based on the mixture of inverse gamma densities which can be used to estimate adaptively densities with very heavy tails, including Cauchy density.

\end{abstract}



\tableofcontents
\end{frontmatter}

\section{Introduction}
\subsection{ Context : posterior concentration rates in Bayesian nonparametric mixture models }

Nonparametric density estimation using Bayesian models with a mixture prior distribution has been used extensively in practice due to their flexibility and available computational techniques using MCMC. In some cases their theoretical properties have been studied, and in particular the asymptotic behaviour of the associated posterior distribution.
Posterior weak consistency has been studied quite systematically in particular by \citep{ghosal:wu}, but posterior concentration rates have been derived only for a small number of kernels.
In the case of density estimation on [0,1] (or any compact interval of $\R$) \citep{rousseau:09} has studied mixtures of Beta densities, and \citep{mcvinishetal:09} have considered mixtures of triangular densities, Gaussian location mixtures have been considered by \citep{ghosal:vdv:01, ghosal:vdv:mixture:07, kruijer:rousseau:vdv:10, ghosal:shen:tokdar, scricciolo:12} and, more generally, power exponential kernels by \citep{kruijer:rousseau:vdv:10,scricciolo:11}. Location scale mixtures have been considered also by \citep{canale:deblasi}. Apart from the latter paper, the posterior concentration rates have been obtained by the above authors are equal to the minimax estimation rate (up to a $\log n$ term) over some collections of functional classes, showing that nonparametric mixture models are not only flexible prior models, but they also lead to optimal procedures, in the frequentist sense.

The above results do not specifically address estimation of  densities on $\R^+ $, they do not cover fat tail densities and the posterior concentration rates have been obtained only under the condition that the densities are uniformly bounded.

In this paper we propose to estimate a possibly unbounded density supported on the positive semiline via a Bayesian approach using a Dirichlet Process mixture of Gamma densities as a prior distribution. The proposed prior distribution does not depend on regularity properties of the unknown density (such as the H\"older exponent) so the resulting posterior estimates are adaptive. Bayesian Gamma mixtures are widely used in practice, for instance, for pattern recognition \citep{GammaMixturesPatternRecogn2} and for modelling the signal-to-noise ratio in wireless channels \citep{GammaMixturesWireless}. An algorithm for implementing a Gamma mixture with unknown number of components as well as aspects of the practical application of this model is given in \cite{GammaMixturesRuggeri2001}.

The main purpose of the paper is to derive the conditions on the Gamma mixture prior  so that the posterior distribution asymptotically concentrates at the optimal rate (up to a log factor) around the true density over smooth classes of densities.
We derive the concentration rate of the posterior distribution when the unknown density belongs to a local H\"older class on $(0,\infty)$ (see the formal definition below) adapting the techniques applied in \citet{ghosal:shen:tokdar}, \citet{rousseau:09} and \citet{kruijer:rousseau:vdv:10} to the proposed mixture of Gamma densities. In particular, we will show that this mixture provides a good approximation for such functions.
Secondly, we investigate the concentration  rate of this posterior distribution for an unknown density on $(0,\infty)$ that can be unbounded at 0, namely   for a density $x^{\alpha-1}h(x)$ for $\alpha \in (0,1)$ in a neighbourhood of 0 and for function $h$ belonging to a locally H\"older class on $(0,\infty)$. A typical example of such behaviour is a Gamma density with the shape parameter between 0 and 1.
One of the novel feature of the paper is that these results hold for densities with polynomial tails. Similar results are obtained for a hierarchical
 Bayesian model based on the mixture of inverse gamma densities which can be used to  estimate adaptively densities with very heavy tails, including the Cauchy density.

For a bounded density, we use the lower bound on the rate of convergence for  estimators of  densities from the local H\"older class $\H(\beta, L)$  which is $n^{-\beta/(2\beta+1)}$ \citep{MM13}.

The paper is organised as follows. In Section~\ref{sec:mainresult} we define the prior distribution and study the concentration rate of the corresponding posterior distribution over an extension of the local H\"older class to possibly unbounded densities. We also discuss the choice of the base measure of the Dirichlet process prior as well as the hyperprior measure on the shape parameter of the Gamma distribution that lead to consistent estimation with the posterior concentration rate equal to the minimax optimal rate of convergence up to a log factor. The prior model based on the mixtures of inverse gammas and the corresponding posterior concentration result are given in Section~\ref{sec:mixtureIG}.  Numerical performance of the estimator is studied on simulated data for bounded and unbounded densities and on real data, with the results presented in Section~\ref{sec:Applications}. The proof of the main result is given in Section~\ref{sec:Proofs}, and the proofs of the auxiliary results are deferred to the appendix.

\subsection{ Setup and Notation}
 Throughout the paper we assume that $X^n = (X_1, \cdots, X_n)$ is an $n$-sample from a distribution with density $f$ on $\mathcal R^+$ with respect to Lebesgue measure. We denote by  $\mathcal F = \{ f \in L_1(\R^+); \, f : \R^+ \rightarrow \R^+; \,\int_{\R^+} f(x) dx = 1 \}$ with $L_1(\R^+)$ denoting the set of measurable and integrable functions on $\R^+$.

 The aim is to estimate the unknown density $f \in \mathcal F$ which we do using a Bayesian approach.
  We construct a prior probability on $\mathcal F$, by modelling $f$ as a mixture of Gamma densities, see \eqref{mixture}  below.
The associated posterior distribution is denoted by $\Pi( \cdot | X^n) $. Let $f_0$ be the true density of the $X_i$'s  and we are interested in determining  the posterior concentration rate $\varepsilon_n = o(1) $ defined by
 \begin{equation*}
 \Pi( B_{\epsilon_n} | X^n )  = 1 + o_{P_{f_0}}(1), \quad B_{\epsilon_n} = \{ f:\, \|f_0 - f \|_1 \leq \varepsilon_n\},
 \end{equation*}
 where $\| \cdot \|_1$ is the $L_1$ norm.

 We denote by $\K(f_1, f_2)$ the Kullback-Leibler divergence between $f_1$ and $f_2$ and
 by $V(f_1, f_2)$ the variance of log-densities ratio:
 \begin{eqnarray*}
 \K(f_1, f_2) &=& \int_{0}^\infty  \log \left(\frac{f_1(x)}{f_2(x)}\right) f_1(x) dx, \\
 V(f_1, f_2) &=& \int_{0}^\infty \left[ \log \left(\frac{f_1(x)}{f_2(x)}\right)\right]^2 f_1(x)dx - \K(f_1, f_2)^2
 \end{eqnarray*}
 and the square of the Hellinger distance by $$D^2_H(f_1, f_2) = \int_{\R^+} (\sqrt{f_1}- \sqrt{f_2})^2(x) dx.$$

Throughout the paper $f(\cdot) \gtrsim g(\cdot)$ (resp. $ f(\cdot) \lesssim g(\cdot) $) means that there exists a positive constant $C$ such that $f(\cdot) \geq C g(\cdot) $ (resp. $f(\cdot) \leq C g(\cdot)) $ and $f(\cdot) \asymp g(\cdot) $ means that $g(\cdot) \lesssim f(\cdot) \lesssim g(\cdot)$.

In the following Section we present the main results of the paper.

\section{Main results}\label{sec:mainresult}

We start with description of the mixture of Gamma distributions which underpins the construction of our prior model on $\mathcal F$.

 \subsection{Prior model : mixtures of Gamma distributions}\label{sec:prior}

We consider the following Gamma mixture types of models:
\begin{equation} \label{mixture}
f_{P, z }(x) = \int_0^\infty g_{z,\epsilon}(x) dP(\epsilon), \quad g_{z,\epsilon}(x) = x^{z-1}e^{ -zx/\epsilon} \left( \frac{ z}{ \epsilon }\right)^{ z } \frac{ 1 }{ \Gamma(z) }.
\end{equation}
We consider $( P, z) \sim \Pi = \Pi_1 \otimes \Pi_z$, where $\Pi_1$ is a probability on the set of discrete distributions over $\R^+$ an $\Pi_z$ is a probability distribution on $\R^{+} = [0,  + \infty)$. We also denote $\R^{+*} = (0, +\infty )$.

Hence the densities are represented by \textit{location} Gamma mixtures, since in the above parametrization $\epsilon$ is the mean of the Gamma distribution with parameters $(z, z/\epsilon)$. This particular parametrization leads to the variance equal to $\epsilon^2 /z$, and as $z$ goes to infinity, the Gamma distribution $(z, z/\epsilon)$ can be approximated by a Gaussian random variable with mean $\epsilon$ and variance $\epsilon^2 /z$.
This allows for precise approximation near 0 and more loose approximation in the tail. This parametrization has also been used in \citet{GammaMixturesRuggeri2001}.

The key to good approximation properties of a continuous density $f$ by the gamma mixtures defined above is
 $$ K_z f (x) \rightarrow f(x) \text{ as } z\to \infty$$
 where operator $K_z$ is defined by
 \begin{equation} \label{operator}
 K_z f(x) =   \int_0^\infty g_{z,\epsilon}(x) f(\epsilon)d\epsilon.
 \end{equation}
We explain in more details in Section \ref{sec:approx:dens}, why mixtures of Gamma distributions as proposed here are flexible models for estimating smooth densities on $\mathbb R^+$. The general idea is that, as in the case of mixtures of Beta distributions in \citet{rousseau:09}  or mixtures of Gaussian distributions  in \citet{kruijer:rousseau:vdv:10}, under regularity conditions with $f$ verifying some H\"older - type condition with regularity $\beta >0$,  one can construct a probability density $f_1$ on $\mathbb R^+$ such that
$$|K_z f_1(x) - f(x)| \lesssim z^{-\beta/2}  , \quad z \rightarrow  +\infty $$
The continuous mixture $K_z f_1$ can then be approximated by a discrete mixture with $O(\sqrt{z}\log z)$ components.

We consider discrete priors on $P$ and priors on $z$ that satisfy the following condition:

\noindent \textbf{ Condition $(\mathbf{\mathcal P})$}:
The prior on $z$, $\Pi_z$ satisfies : for some constants $c\geq  c'>0 , c_0 >0$ and $\rho_z \geq 0$,
\begin{equation}
\begin{split}
\Pi_z([x, 2x]) &\gtrsim e^{-c \sqrt{x} (\log x)^{\rho_z}}, \quad \Pi_z ([x, +\infty )]\lesssim e^{-c' \sqrt{x} (\log x)^{\rho_z}} \quad  \mbox{ as } x\rightarrow +\infty, \\
 \Pi_z( [0,x]) & \lesssim x^{c_0} \quad \text{ for } x \rightarrow  0,
\end{split}
\label{eq:PriorZ}
\end{equation}

We consider either of these two types of prior on $P$:
\begin{itemize}
\item \textbf{Dirichlet Prior} of $P$: $P\sim DP(m, G)$ where  $DP(m,G)$ denotes the Dirichlet Process with mass $m>0$ and base probability measure $G$ having  positive and continuous density $g$ on $\R^{+*}$ satisfying:
 \begin{equation} \label{cond:priordens}
  x^{a_0} \lesssim g(x) \lesssim x^{a_0'} \mbox{ as } x \rightarrow 0, \,\, \& \,\,   x^{-a_1} \lesssim g(x) \lesssim x^{-a_1'}  \mbox{ as }  x \rightarrow +\infty
  \end{equation}
  for some $-1< a_0' \leq a_0$ and $ 1 < a_1' \leq a_1$ .
\item \textbf{Finite mixture } :
\begin{gather*} P (d\epsilon) = \sum_{j=1}^K p_j \delta_{\epsilon_j}(d\epsilon) , \quad K \sim \pi_K,  \quad \epsilon_j \stackrel{iid}{\sim} G,\\
   e^{- c_1 k (\log k)^{\rho_2}}\gtrsim  \pi_K(k) \gtrsim e^{- c_2 k (\log k)^{\rho_2}} \, \forall k\geq 2,\\
 (p_1, \cdots, p_k ) \sim \mathcal D(\alpha_1, \cdots, \alpha_k),
\end{gather*}
 with $G$ satisfying \eqref{cond:priordens}, $\rho_2 \geq 0$, $0 < c_1 \leq c_2 < +\infty$ and there exists $\bar m$ such that
  $$ \sum_{i=1}^k \alpha_i \leq \bar m k , \quad \sum_{i=1}^k (-\log \alpha_i)_+ \leq \bar m k \log k\,\,  \forall k\geq 2.$$
Note that $(\alpha_j, j\leq k)$ possibly all depend on $k$, but we omit the additional index $k$ to simplify the notation. 
\end{itemize}

Condition $(\mathbf{\mathcal P})$ is quite mild. It is satisfied for instance for $\sqrt{z} $ following a $Gamma $ distribution, in which case $\rho_z=0$. The prior condition on the base measure \eqref{cond:priordens} imposes fat tails on $G$. It is satisfied for instance if
$g(x) \propto x^{a_0} ( 1 + x^2)^{-a_0- 1} $ with $a_0 >0$.

 Note that, it appears from the proofs that if $g(x) \leq e^{-c x}$ then posterior concentration rates would remain unchanged over the functional class described below, assuming that the true density $f$ also has exponential type tails. Hence, when estimating such densities, Gamma densities satisfy the conditions required for $G$  but inverse Gammas do not.

\subsection{Functional classes} \label{sec:fnclass}

In this paper we are interested in estimating densities that are possibly unbounded at 0. To construct a class of such functions over which posterior concentration rates are derived, first we need a class of bounded functions.

Let $\mathcal P ( \beta, L(\cdot), \gamma, C_0, C_1, e, \Delta ) $, $\beta>0$,  $\gamma\geq 0$ be the set of such functions  $f: \, \R^+ \to [0,\infty)$ which are $r$ times continuously differentiable with $r = \lceil \beta \rceil - 1$ and which  satisfy for all $x \in \R_+$ and $y$: $y>-x$ and $|y|\leq \Delta$,
 \begin{equation}\label{cond:taylor}
\left| f^{( r )}(x+y) - f^{( r )} (x) \right| \leq L(x) |y |^{\beta - r } ( 1 + |y|^{\gamma}), \quad f(x) \leq C_0,
\end{equation}
 defining $r_0 = \lceil \beta /2 \rceil - 1$ for $\beta >2$ and $r_0=0$ if $\beta \leq 2$,
\begin{eqnarray}\label{moment:func}
\int_0^\infty \left(\frac{ x^j |f^{(j)}(x)|  }{ f(x) } \right)^{ (2 \beta + e)/j} f(x) dx &\leq& C_1 ,\quad  j \leq r;  \notag \\
 \int_0^\infty \left(\frac{ (1  + x^{\gamma+ 2r_0}) x^\beta L(x)   }{  f(x) }\right)^{2} f(x) dx &\leq& C_1
\end{eqnarray}
for some $ e >0$. In the above definition  $L(.)$ is a  fixed positive function from $\mathbb R^+$ to $\mathbb R^+$.

We also consider classes of densities unbounded around zero that are defined as follows.
 Let  $\mathcal P_{\alpha}( \beta, L(\cdot), \gamma, C_0, C_1, e, \Delta)$, $\alpha \in (0,1]$, $\beta>0$,  $\gamma\geq 0$,  be the set of such functions $f: \,\R^+ \to [0,\infty)$ such that
 function $h(x) :=  x^{-(\alpha-1)} f(x)$ satisfies the following conditions:
 \begin{itemize}
 \item[1)] function $h$ is $r$ times continuously differentiable with $r = \lceil \beta \rceil - 1$ and such that for all $x>0$, $y>-x$ and $|y|\leq \Delta$,
     \begin{equation}\label{cond:tayloralpha}
\left| h^{( r )}(x+y) - h^{( r )} (x) \right| \leq L(x) |y |^{\beta - r } ( 1 + |y|^{\gamma}), \quad h(x) \leq C_0
\end{equation}
\item[2)] for some $ e >0$,
\begin{gather} \label{integralHolderalpha}
\int_0^\infty \left(\frac{ x^j |h^{(j)}(x)| }{ h(x) } \right)^{ (2 \beta + e)/j} x^{\alpha-1}h(x) dx \leq C_1 ,\quad  j \leq r;  \notag\\
 \int_0^\infty \left(\frac{  L(x)x^{\beta}(1+ x^{\gamma+ 2r_0}) } { h(x) }\right)^2 x^{\alpha-1}h(x) dx \leq C_1,
\end{gather}
where $r_0 = \lceil \beta /2 \rceil - 1$ for $\beta >2$ and $r_0=0$ if $\beta \leq 2$.
\end{itemize}

Note that  in the case $\alpha = 1$, we recover the first functional class, namely
$$\mathcal P_{1}( \beta, L(\cdot), \gamma,C_0, C_1, e, \Delta) = \mathcal P( \beta, L(\cdot), \gamma,C_0, C_1, e, \Delta).$$

The rationale behind the functional class  $\mathcal P_{\alpha}( \beta, L(\cdot), \gamma, C_0, C_1,e,\Delta)$  comes from the following lemma.
\begin{lemma}\label{lem:representation}
 For any $f\in \mathcal P_{\alpha}( \beta, L(\cdot), \gamma, C_0, C_1, e, \Delta )$ and $x>0$,
$$
K_z f(x) = x^{\alpha-1} K_{z+1-\alpha} h(x/C_z) (1+O_z)
$$
where $h(x) = x^{1-\alpha}f(x)$,  $C_z = 1 +(1-\alpha)/z$ and $O_z=\frac{z^{\alpha}\Gamma(z-\alpha)}{\Gamma(z)}-1=O(1/z)$ for large $z$.
\end{lemma}
Lemma \ref{lem:representation} is a consequence of Lemma~\ref{lem:ApproxAlpha} which is given in Appendix~\ref{sec:ProofsAdjAlpha}.

\begin{remark}\label{rk:functionalclass}
\begin{enumerate}
\item Note that in the above functional class $\mathcal P ( \beta, L(\cdot), \gamma, C_0, C_1, e, \Delta ) $, $f$ is bounded from above but is allowed converge to 0 as $x$ goes to 0. Interestingly if $f(x) = x^{\tau } h(x)$ for any $\tau \geq 0$ and $h$ satisfies \eqref{moment:func} then condition \eqref{moment:func} is satisfied by $f$.

\item  If  $\log f$ is locally H\"older, with suitable integrability conditions,  then  $f(x) \in \mathcal P ( \beta, L(\cdot), \gamma, C_0, C_1, e, \Delta ) $. Similarly, the same holds for  $h(x)$ bounded from below if $f(x)=x^{\alpha-1} h(x) \in \mathcal P_\alpha ( \beta, L(\cdot), \gamma, C_0, C_1, e, \Delta ) $.



\item The moment condition \eqref{moment:func} is satisfied for a number of densities on the positive semiline (see Appendix~\ref{sec:VerifCond} for details):
\begin{enumerate}
\item   for the Weibull distribution with density $f(x) = C_{a,b} x^{a-1} e^{- x^{b}}$, $a, b> 0$, with $\beta=b$ for $a\in (0,1]$; for $a\geq 1$, $\log f$ is H\"older with exponent $b$;
 \item for  folded Student t distribution $f(x) = c_\nu (1+x^2)^{-(\nu+1)/2} I(x>0)$ with $\nu\geq 1$;

 \item  for a Frechet-type distribution with density  $f(x)= c x^{-b-1} e^{-x^{-b}}$, $b>0$. In this case,  $\lim_{x\to 0}f(x) =0$ so we can take $\alpha=1$, $\gamma=1$ and $\beta \leq 1$.
\end{enumerate}

\end{enumerate}

\end{remark}

\subsection{Posterior concentration rates}

\subsubsection{Posterior concentration rate for the mixture of Gammas} \label{sec:post}

Similarly to location  mixtures of Gaussian densities, mixtures of Gamma densities provide a flexible tool to approximate smooth  densities on $\R^+$, and using the representation of Lemma \ref{lem:representation}, to approximate smooth but possibly unbounded densities. The posterior concentration rates are presented in the following Theorem.
In addition to the regularity conditions induced by the functional classes defined above we will need the following tail assumption :
 \begin{equation} \label{tailcond}
 \exists \rho_1 >0 \, \& \, C_2 >0: \,\, \int_{x}^\infty y^2 f_0(y) dy  \leq C_2 (1+x)^{-\rho_1}.
  \end{equation}
We denote by $\mathcal T(\rho_1, C_2)$ the set of densities satisying \eqref{tailcond}.

\begin{theorem}\label{th:rate}
Consider the prior defined in Section \ref{sec:prior} and
assume that $\mathbf{X}^n = (X_1, \cdots, X_n)$ is a sample of independent observations identically distributed according to a probability $P_0$ on $\R^+$ having density  $f_0$ with respect to Lebesgue measure.

Then, for any  $\alpha_0>0$, $\beta_1\geq\beta_0>0$, $L(\cdot)$, $\gamma$,  $C_0$, $C_1$,  $e$, $\Delta$, $C_2$ and $\rho_1$ , there exists $M>0$ such that
\begin{equation} \label{eq:rate}
\sup_{\alpha \in [\alpha_0,1]} \sup_{\beta \in [\beta_0,\beta_1]} \sup_{f_0 \in  \mathcal Q_{\alpha,\beta}(\cdots)} E_{P_0}\left[ \Pi\left( \| f - f_0\|_1 > M \frac{(\log n)^{q}}{n^{\beta/(2\beta+1)}}| \mathbf{X}^n \right) \right] = o(1)
\end{equation}
with $q= (5 \beta + 1)/(4\beta+2)$ if $\rho_z \leq 5/2$ and $q = ( 2 \rho_z \beta +1)/(4\beta+2)$ if $\rho_z > 5/2$ and
$$\mathcal Q_{\alpha,\beta} (\cdots) =  \mathcal P_{\alpha}( \beta, L(\cdot), \gamma,  C_0, C_1, e, \Delta) \cap \mathcal T(\rho_1, C_2).$$
\end{theorem}

Theorem \ref{th:rate} is proved in Section \ref{sec:pr:thrate}.

Note that in \eqref{eq:rate}, $f_0$ is supposed to satisfy \eqref{tailcond} so that the supremum in $f_0$ is taken over the intersection of

Note 	also that the tail assumption \eqref{tailcond}  is much weaker than  the tail condition considered for mixtures of Gaussians as in \citep{kruijer:rousseau:vdv:10} or \citet{ghosal:shen:tokdar} where exponential decay in the form $e^{- c |x|^\tau}$ is assumed for some $\tau >0$. Because the posterior concentration rates obtained in \citep{kruijer:rousseau:vdv:10} or \citet{ghosal:shen:tokdar} are upper bounds, it is not clear that the exponential tail condition was necessary. It seems however that because Gamma densities have fatter tails than Gaussian densities, they allow for approximations of densities with fatter tails.

The theorem is proved following the approach of \citet{ghosal:ghosh:vdv:00}, so that first we construct an approximation of the true density $f_0$ by a continuous mixture of Gammas in the form $K_zf_1$ for some density $f_1$ close to $f_0$ and then approximate the continuous mixture by a discrete and finite mixture. This allows us to control the prior mass of Kullback-Leibler neighbourhoods. The tail condition \eqref{tailcond} is used in this second step, similarly as the exponential tail condition used with Gaussian mixtures in  \citep{kruijer:rousseau:vdv:10} or \citet{ghosal:shen:tokdar}. It is to be noted however that the integrability conditions \eqref{integralHolderalpha} in the  definition of the functional classes $\mathcal P_\alpha ( . ) $ implicitly  induce some tail constraints on $f_0$, see  Remark  \ref{rk:functionalclass}.

\subsubsection{Mixture of inverse Gamma distributions}\label{sec:mixtureIG}

Although \eqref{tailcond} is a rather mild condition, it excludes fat tail distributions such as the folded Cauchy density whose density at infinity behaves like $x^{-2} $. The Bayesian model proposed here can be adapted to estimating fat tail densities at infinity as follows.

Let  $X_i \sim f_0$ with $f_0(x) \lesssim x^{a+1}$ for small $x$ and some $a>0$ and $f_0(x) \lesssim x^{-b-1}$ for some $0< b\leq 2$ when $x$ goes to infinity. Then the density of $1/X_i$, $\finv_0(y) =y^{-2} f_0(1/y) \lesssim y^{-a-1}$ when $y$ goes to infinity and $\finv_0(y) \lesssim y^{b-1}$ when $y$ goes to 0. Hence $\finv_0$ satisfies the tail conditions both at 0 and infinity assumed in Theorem~\ref{th:rate}.
 Since $||\finv-\finv_0||_1 = ||f-f_0||_1$, Theorem~\ref{th:rate} implies that density $\finv(x)$ can be estimated using the appropriately adapted prior, such that the corresponding prior on $f$ satisfies the conditions stated in the theorem, with the same rate of convergence.

The prior for estimating $f(x) = x^{-2} \finv(1/x)$ is a mixture of inverse Gammas:
\begin{equation} \label{mixtureIG}
 f_{P, z }(x) = \int_0^\infty \bar{g}_{z, \xi}( x) d\bar{P}(\xi)
\end{equation}
where
$$
\bar{g}_{z, \xi}( x) = x^{-2} g_{z,1/\xi}(1/x) = x^{-z-1}e^{ -z\xi/x} \left(  z \xi \right)^{ z } \frac{ 1 }{ \Gamma(z) }
$$
which is the density of an inverse Gamma distribution. 

\noindent \textbf{ Condition} $(\bar{\mathbf{\mathcal P}})$:

 The hyperprior is $(\bar{P}, z) \sim \bar\Pi = \bar\Pi_1 \otimes \Pi_z$, where $\Pi_z$ is a probability distribution on $\R^{+}$ satisfying condition \eqref{eq:PriorZ} and $\bar\Pi_1$ is a probability on the set of discrete distributions over $\R^+$ satisfying either of these two types of prior on $P$:
\begin{itemize}
\item \textbf{Dirichlet Prior} of $\bar P$: $\bar P\sim DP(m, \bar G)$ where  $DP(m,\bar G)$ denotes the Dirichlet Process with mass $m>0$ and base probability measure $\bar G$ having  positive and continuous density $\bar g$ on $\R^{+*}$ satisfying:
 \begin{equation} \label{cond:priordensInverse}
y^{\bar a_0} \lesssim  \bar g(y)  \lesssim y^{\bar a_0'}  \mbox{ as }  y \rightarrow 0, \,\, \& \,\, y^{-\bar a_1} \lesssim  \bar g(y) \lesssim y^{-\bar a_1'} \mbox{ as } y \rightarrow +\infty
  \end{equation}
  for some $ -1 < \bar a_0  \leq \bar a_0'$ and $1< \bar a_1  \leq \bar a_1'$ (note that this corresponds to the density $x^{-2} \bar g(1/x)$ satisfying conditions \eqref{cond:priordens} with $\bar a_0 = a_1-2$ and $\bar a_1 = a_0+2$).

\item \textbf{Finite mixture } :
\begin{gather*}
\bar P (d\xi) = \sum_{j=1}^K p_j \delta_{\xi_j}(d\xi) , \quad K \sim \pi_K,  \quad \xi_j \stackrel{iid}{\sim} \bar G,\\
  e^{- c_1 k (\log k)^{\rho_2}}\gtrsim  \pi_K(k) \gtrsim e^{- c_2 k (\log k)^{\rho_2}} \, \forall k\geq 2,\\
  (p_1, \cdots, p_k ) \sim \mathcal D(\alpha_1, \cdots, \alpha_k),
 \end{gather*}
 with $\bar G$ satisfying \eqref{cond:priordensInverse}, $\rho_2 \geq 0$, $0 < c_1 \leq c_2 < +\infty$ and there exists $\bar m$ such that
  $$\forall k \geq 1, \quad  \sum_{i=1}^k \alpha_i \leq k \bar m, \quad \sum_{i=1}^k (-\log \alpha_i)_+ \leq \bar m  k \log k \,\,  \forall k\geq 2.$$
\end{itemize}



In particular, we have similar approximation properties for $f$ in the following class:
$$
\bar{\mathcal P}_{\alpha}( \beta, L(\cdot), \gamma, C_0, C_1, e, \Delta)  = \{f: \, x^{-2} f(1/x) \in \mathcal P_{\alpha}( \beta, L(\cdot), \gamma, C_0, C_1, e, \Delta)\},
$$
since for all $x>0$, as $z\to \infty$, with $\finv(x) = x^{-2} f(1/x)$,
$$
\bar{K}_z f (x) \stackrel{def}{=} \int_0^\infty \bar{g}_{z, \xi}( x) f (\xi) d\xi = \frac 1 {x^{2}} \int_0^\infty g_{z, \epsilon}(1/x) \finv(\epsilon) d\epsilon \to \frac 1 {x^{2}}\finv(1/x) = f (x).
$$

We summarize the result in the following corollary.
\begin{corollary}\label{th:rateIG}
Consider the prior defined by \eqref{mixtureIG} that satisfies condition $(\bar{\mathcal P})$, and
assume that $\mathbf{X}^n = (X_1, \cdots, X_n)$ is a sample of independent observations distributed according to a probability $P_0$ on $\R^+$ having density  $f_0$ with respect to Lebesgue measure.

Then, for any  $\alpha_0>0$, $\beta_1\geq\beta_0>0$, $L(\cdot)$, $\gamma$,  $C_0$, $C_1$, $e$, $\Delta$, $C_2$ and $\rho_1$ , there exists $M>0$ such that
\begin{equation*}
\sup_{\alpha \in [\alpha_0,1]} \sup_{\beta \in [\beta_0,\beta_1]} \sup_{f_0 \in  \bar{\mathcal Q}_{\alpha,\beta}(\cdots)} E_0\left[ \Pi\left( \| f - f_0\|_1 > M n^{-\beta/(2\beta+1)}(\log n)^{q}| \mathbf{X}^n \right) \right] = o(1)
\end{equation*}
with $q= (5 \beta + 1)/(4\beta+2)$ if $\rho_z \leq 5/2$ and $q = ( 2 \rho_z \beta +1)/(4\beta+2)$ if $\rho_z > 5/2$ and
$$\bar{\mathcal Q}_{\alpha,\beta} (\cdots) = \{f: \, \,  x^{-2} f(1/x) \in \mathcal P_{\alpha}( \beta, L(\cdot), \gamma,  C_0, C_1, e, \Delta) \cap \mathcal T(\rho_1, C_2)\}.$$
\end{corollary}
 Note that the folded Cauchy density satisfies the conditions for $f_0(x)$ required in the corollary.

\subsection{Approximation of densities by Gamma mixtures} \label{sec:approx:dens}

As in \citet{rousseau:09, kruijer:rousseau:vdv:10}, one of the key elements in the proof of Theorem \ref{th:rate}  is the approximation of a smooth density $f$ by a continuous Gamma mixture of the form $ K_z f_1$ where $f_1$ is a smooth function close to $f$ which is of independent interest. Similarly to \citet{rousseau:09,kruijer:rousseau:vdv:10}, $f_1$ is constructed iteratively to be able to adapt to the smoothness of $f_0$. The general idea is that $K_z f_0(x) $ is a good approximation of $f_0$ if $f_0$ has smoothness  $\beta\leq 2$, as in the Gaussian case, because the Gamma density $g_{z,\epsilon}$ behaves like a Gaussian density when $z$ goes to infinity. To approximate $f_0$ with the correct order $z^{-\beta/2}$ for $\beta >2$, we need to correct for the error $K_z f_0 - f_0$ and replace $K_z f_0$ by $K_z f_1$ where construction of $f_1$ takes into account $K_z f_0 - f_0$. Thus, we iterate until the approximation error $K_z f_k - f_k$ is of the required order, $z^{-\beta/2}$.

The above approximation scheme is described in the following proposition:
 \begin{proposition}\label{lem:approx:1}
 For all $z>1-\alpha$, for all $f \in \mathcal P_{\alpha} ( \beta, L(\cdot), \gamma,  C_0, C_1, e, \Delta)$, there exist $d_j \in \R$, $j=1, \cdots, 2 r_0 $ such that the function $\tilde f_{\beta,\alpha}(x) = C_z f_{\beta,\alpha}(C_z x)$ with $C_z = (z -\alpha+1)/z$ and
$$
f_{\beta, \alpha}(x) = f(x) - x^{\alpha-1}\sum_{j=1}^{r}\frac{ d_j(z) }{ j !z^{j/2} } x^j [x^{1-\alpha}f(x)]^{(j)}, \quad d_j(z) = d_j + O(1/z),
$$
  satisfies
    \begin{equation} \label{pointwise:approx}
    \left| K_z \tilde{f}_{\beta,\alpha} (x) - f(x) \right| \leq  z^{-\beta/2} R(x), \quad \int_0^\infty \frac{ R(x)^2 }{ f(x) }dx  \leq C_R , \end{equation}
  where  $C_R$ depends on $\alpha, \beta, L(\cdot), \gamma,  C_0, C_1, e, \Delta$ only.
Moreover, the probability density
  \begin{equation} \label{fbeta}
 \bar f_\beta = c_\beta \left( \tilde f_{\beta,\alpha}\1_{\tilde f_{\beta,\alpha}\geq \tilde f /2} +\frac{ \tilde f}{2}\1_{\tilde f_{\beta,\alpha}< \tilde f /2}\right)
 \end{equation}
with $\tilde f (x) = C_z f(C_z x) $ for all $x \in \R^+$, satisfies
\begin{equation}\label{Hell:contdens}
D_H(K_z \bar f_\beta, f) \leq B z^{-\beta/2}, \quad \forall f \in \mathcal P_{\alpha} ( \beta, L(\cdot), \gamma,  C_0, C_1, e, \Delta)
\end{equation}
where $B$ depends only on $\alpha, \beta, L(\cdot), \gamma,  C_0, C_1, e, \Delta$.
\end{proposition}

Note that in  the approximations \eqref{pointwise:approx} and  \eqref{Hell:contdens}, when $\alpha \in [\alpha_0, 1]$ with $0 < \alpha_0 \leq 1$, then $C_R$ and $B$ can be chosen depending on $\alpha_0$ and not $\alpha$.
Proposition  \ref{lem:approx:1} is proved in Section \ref{sec:proof:approx1}.

\section{Numerical results}\label{sec:Applications}

\subsection{Prior model}

In the following sections we will fit a Bayesian model to the data with following Dirichlet Process prior:
\begin{eqnarray}\label{eq:DefaultPrior}
f(x) &=& \sum_j p_j g_z(x|\epsilon_j) , \quad \epsilon_j\stackrel{iid}{\sim} G, \quad p_j = V_j\prod_{i<j}(1-V_i), \quad V_i \sim \mbox{Beta}(1, m),\notag\\
G(x)  &\propto& [x^a I_{x\leq 1}+ x^{-a}I_{x>1}], \quad a>1,\\
\Pi_z: && \sqrt{z} \sim \Gamma(b,c), \quad b,c>0. \notag
\end{eqnarray}
It is easy to check that Condition $(\mathbf{\mathcal P})$  holds for this prior.
 We use default choices of free parameters $m=1$, $a=2$ and $b=c=1$, however we check sensitivity with respect to these parameters.

To sample from the posterior we use the slice sampling algorithm of \citet{Kalli2011}. Introducing the auxiliary variables $\mathbf u= (u_1, \cdots, u_n)$ the uniform random truncation variables  and $\mathbf c = ( c_1, \cdots, c_n)$ the allocation variables so that the full likelihood is written as
$$L_n(\mathbf X^n , \mathbf u, \mathbf c) = \prod_{i=1}^n\frac{ \1_{u_i\leq p_{c_i}}}{  p_{c_i}} g_{z, \epsilon_{c_i}}(X_i) p_{c_i}, \quad p_j = V_j \prod_{l<j} (1-V_l)$$
allows to use a Gibbs sampler algorithm based on the following full conditional distributions
\begin{itemize}
\item $$[\epsilon_j | \cdots] \stackrel{ind}{\sim } G(\epsilon_j) e^{-z S_j/\epsilon_j} \epsilon_j^{-z n_j}  , \quad S_j = \sum_{c_i=j}X_i, \quad n_j = \sum_{i=1}^n\1_{c_i=j}$$
\item $$[V_j | \cdots \mbox{exclude} \, u]  \stackrel{ind}{\sim } \mbox{Beta}( n_j+1, \sum_{l>j}n_l + m)$$
\item $$[u_i | \cdots ]  \stackrel{ind}{\sim } \mathcal U(0,p_{c_i})$$
\item $$p[c_i=j| \cdots ] \propto \1_{u_i\leq p_{j}}g_z(X_i|\epsilon_j)$$
\item $$[z| \cdots ] \propto \frac{\pi_z(z) z^{nz}}{ \Gamma(z)^n}e^{ -z \sum_j S_j/\epsilon_j + z \sum_i \log X_i - z\sum_{j: n_j>0}n_j \log \epsilon_j} $$
\end{itemize}
where the full conditional distribution of $z$ is sampled using a Metropolis-Hastings step with proposal
\begin{equation}\label{eq:ProposalZ}
\Gamma\left((b+n)/2, \sum_{j: n_j>0} \frac{ S_j}{\epsilon_j} - n - \sum_i \log X_i+ \sum_{j: n_j>0}n_j \log \epsilon_j\right).
\end{equation}
To improve convergence when sampling from  the conditional distribution of $z$ at iteration $i$, we also use a proposal which is a mixture of the gamma density \eqref{eq:ProposalZ} with weight $1-w_z$ and  of $\Gamma( z^{(i-1)} B_z,B_z)$ with a small weight $w_z$ and a large $B_z$ where $z^{(i-1)}$ is the value of $z$ sampled at iteration $i-1$. We use the default values $w_z=0.01$ and $B_z=10$.

Another Metropolis-Hastings step is sampling from the full conditional distribution of $\epsilon_j$ which uses proposal $IG(a+z^{(i-1)} n_j, z^{(i-1)} S_j)$ at iteration $i$ for $j$: $n_j>0$.

\subsection{Simulations}

  \begin{figure}
 \centering
 \includegraphics[width=5cm , height= 5cm]{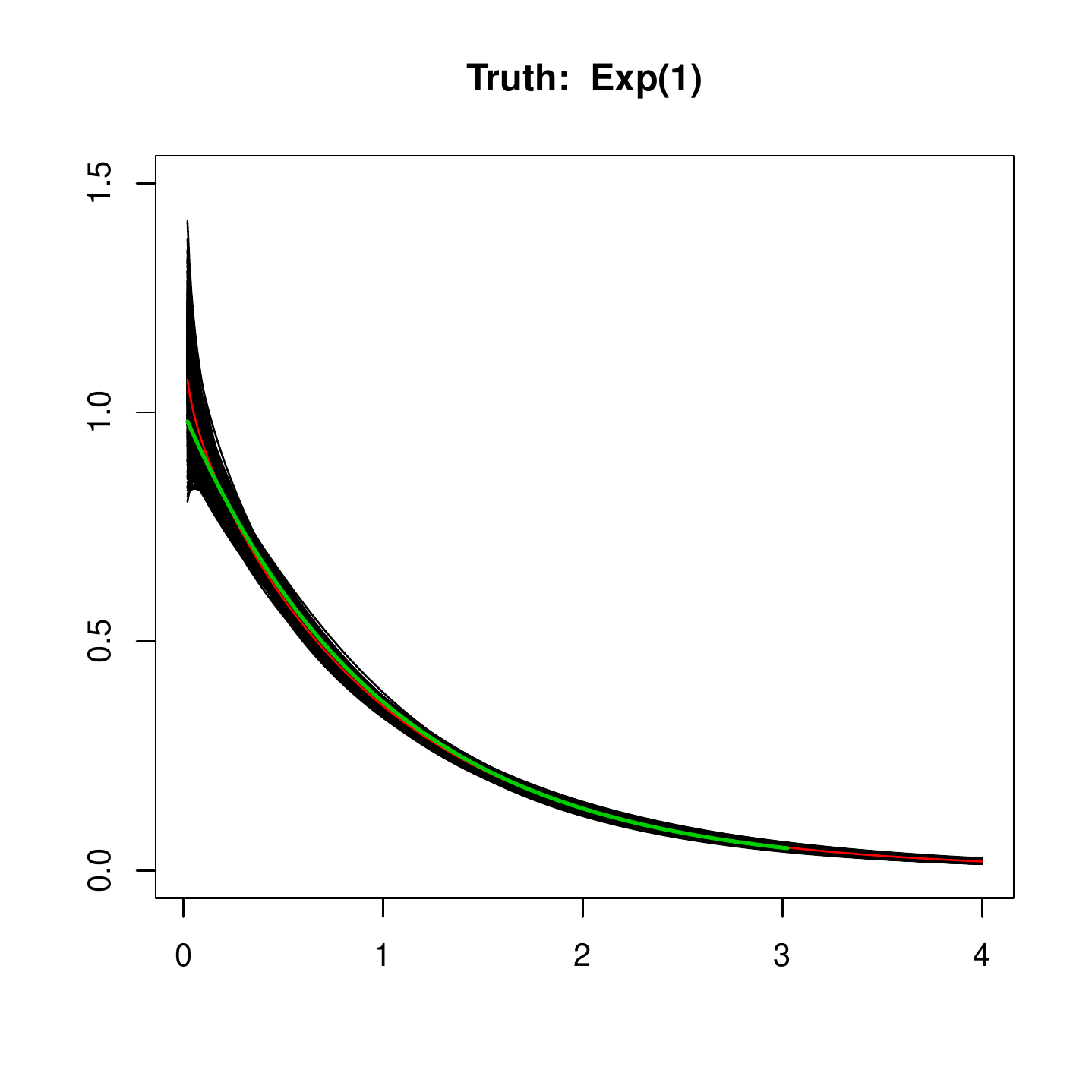}
 \includegraphics[width=5cm , height= 5cm]{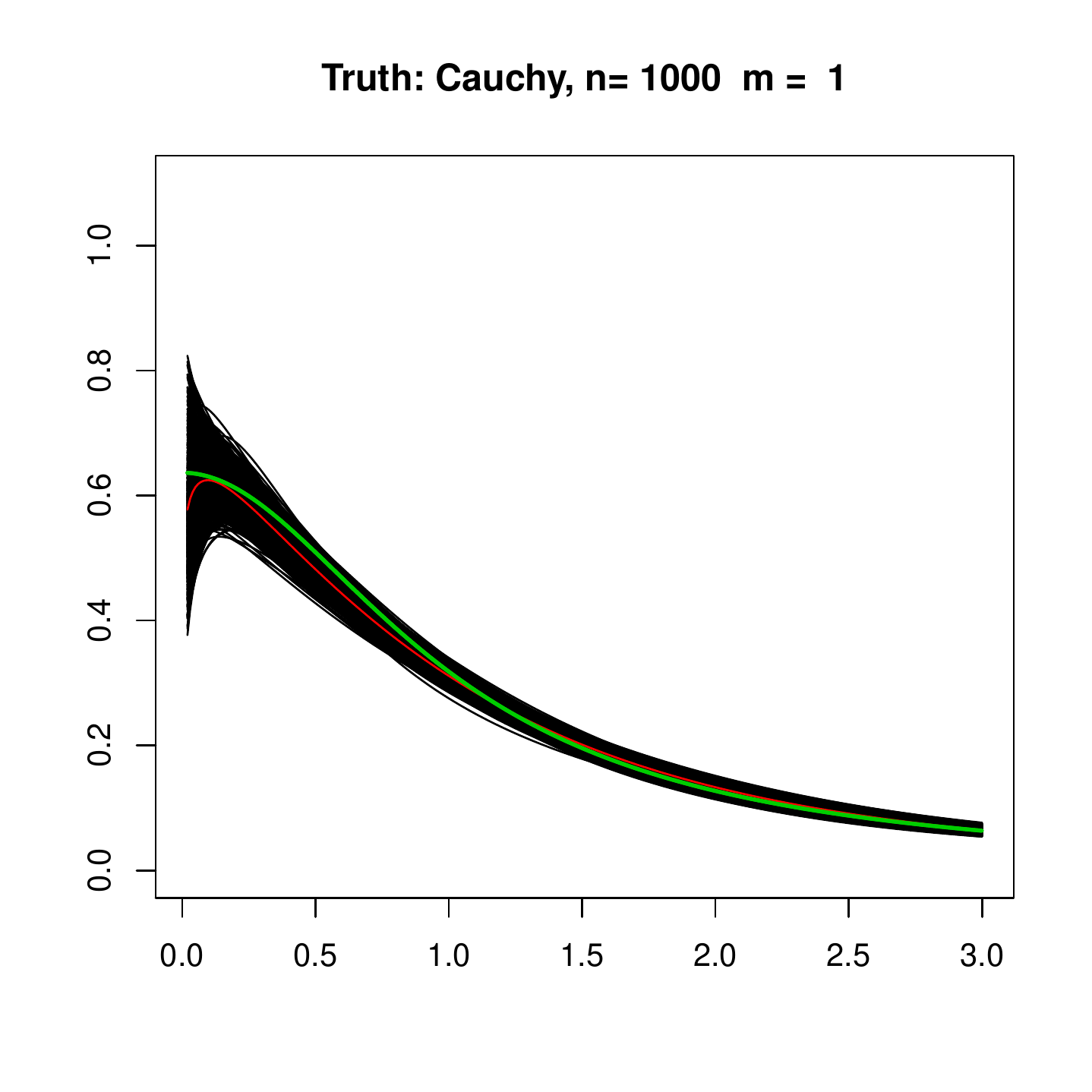}
\caption{Left: exponential density. Right: folded Cauchy density}
\label{fig:ExpCauchy}
 \end{figure}

  \begin{figure}
 \centering
 \includegraphics[width=5cm , height= 5cm]{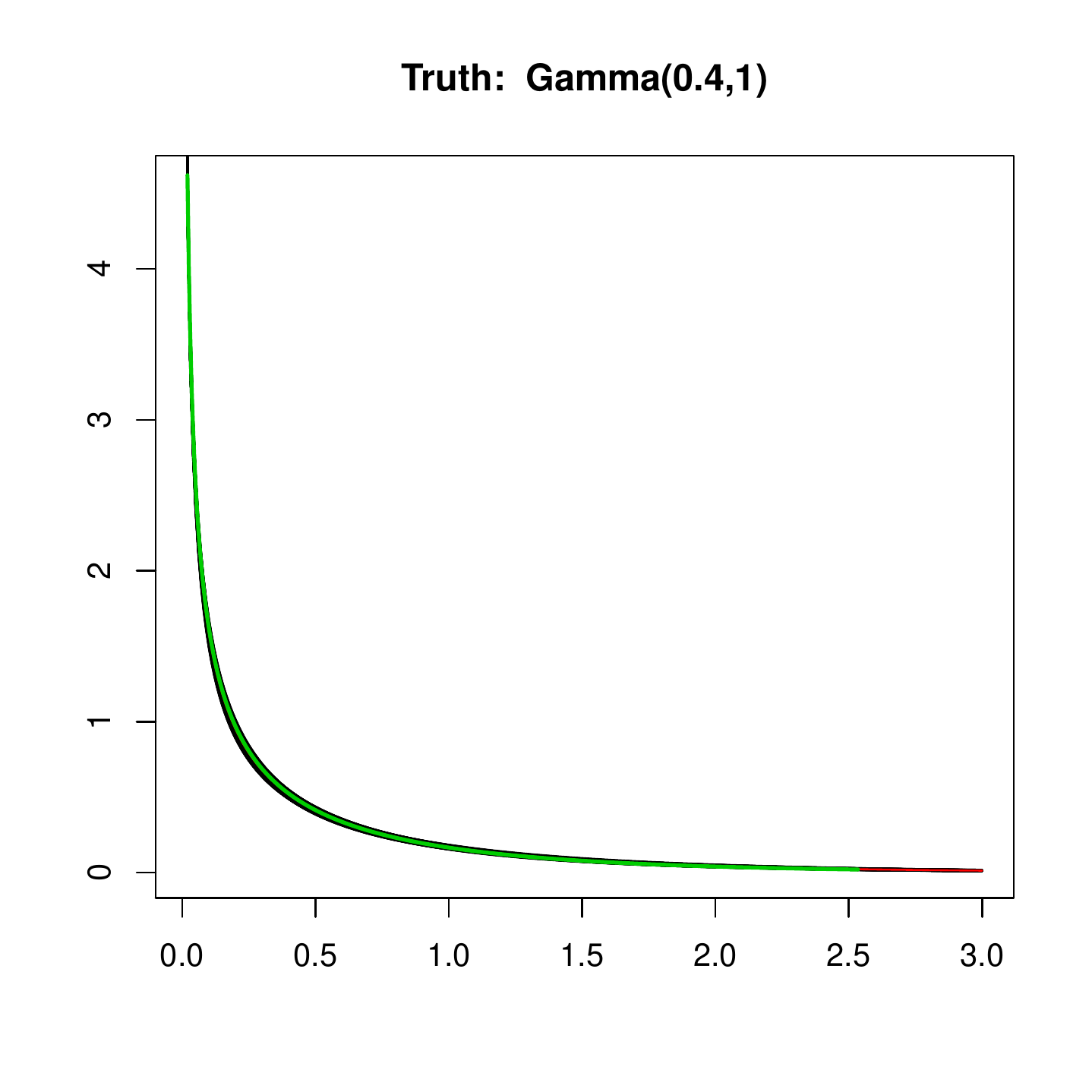}
 \includegraphics[width=5cm , height= 5cm]{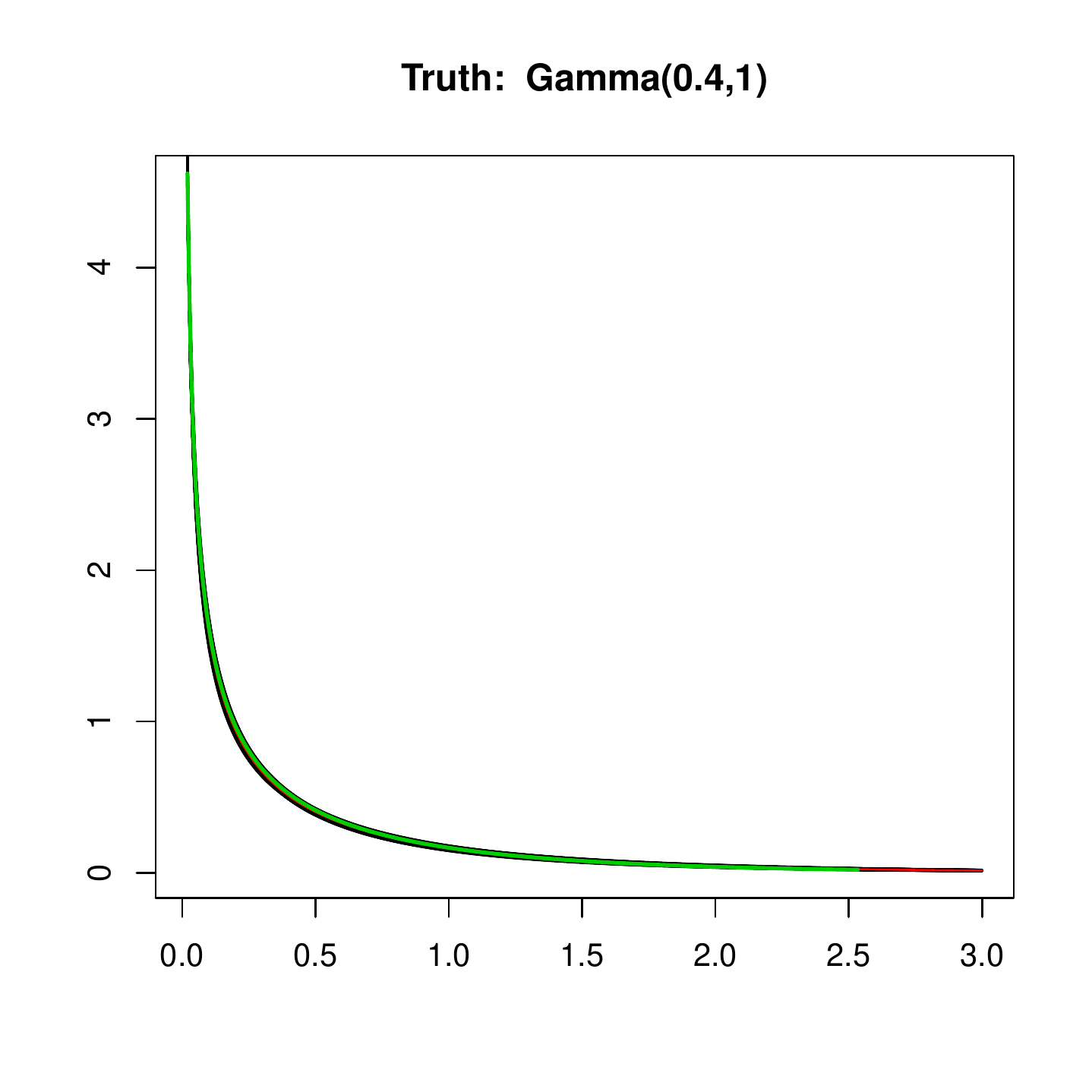}
\caption{$f_0(x) = Gamma(0.4, 1)$. Left: m=0.1, right : m=1}
\label{fig:GammaUnbounded}
 \end{figure}

  \begin{figure}
 \centering
 \includegraphics[width=5cm , height= 5cm]{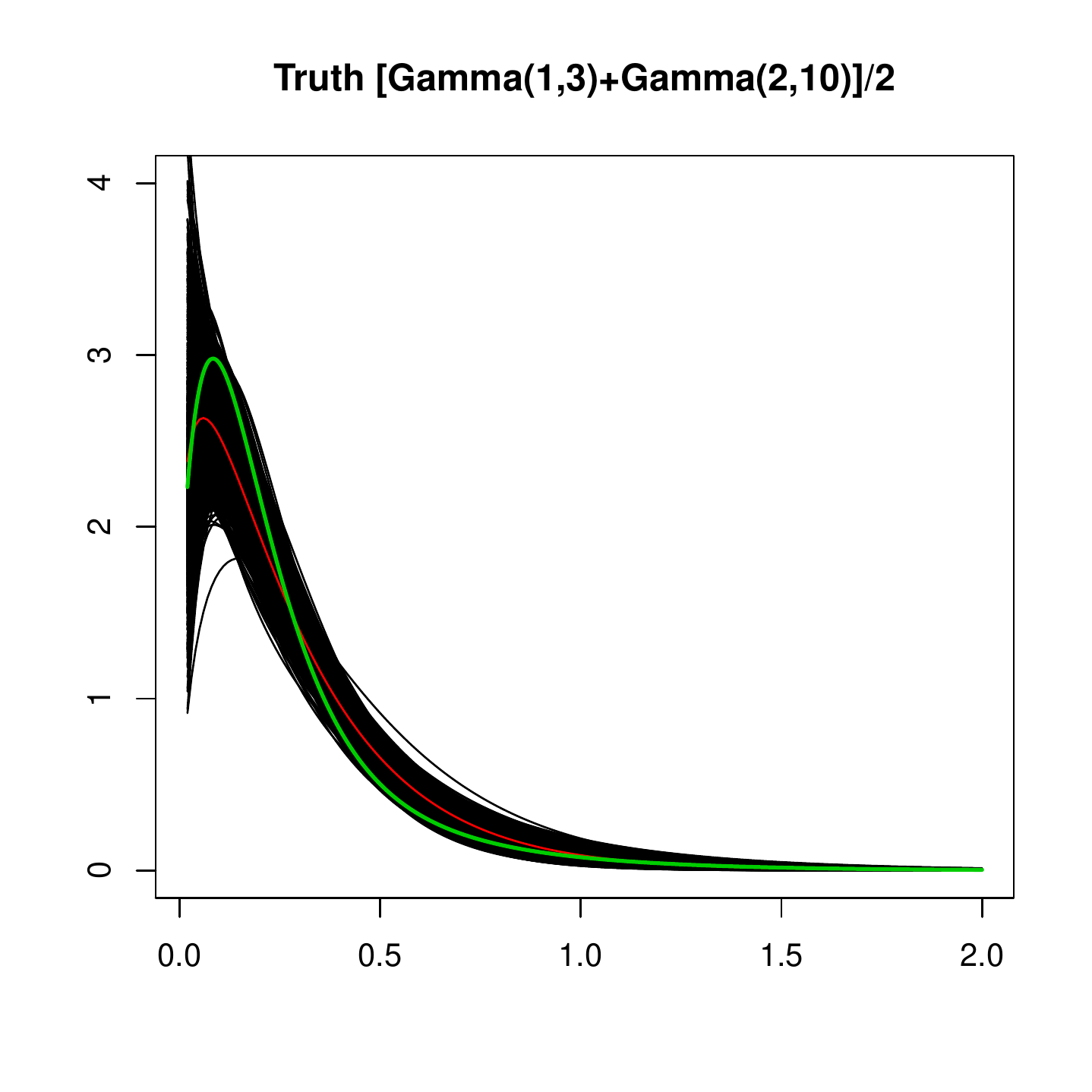}
 \includegraphics[width=5cm , height= 5cm]{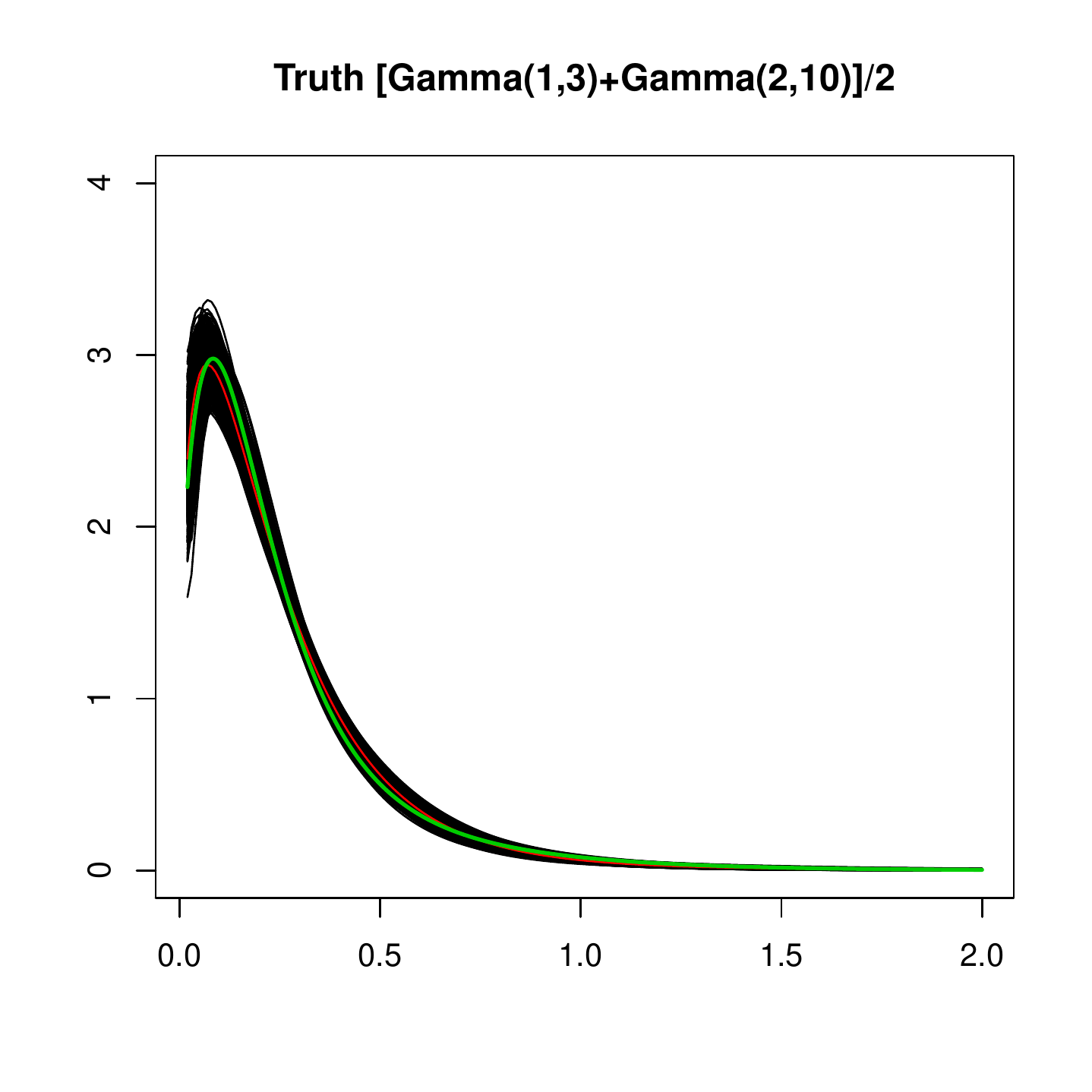}
\caption{Density:  0.5 Gamma(1,3) + 0.5 Gamma(2,10). Left: n=100, right : n= 1000}
\label{fig:MixtureDiffZ}
 \end{figure}

We simulate $n=1000$ observations from the true density $f_0$, and fit a Bayesian model with prior \eqref{eq:DefaultPrior}. We considered the following true densities.
\begin{enumerate}
\item  Exponential: $f_0(x) = e^{-x} $, $x>0$  (Figure \ref{fig:ExpCauchy}).
\item  Folded Cauchy: $f_0(x)  = \frac{ 2 }{\pi} \frac 1 {(1+x^2)}$, $x>0$   (Figure \ref{fig:ExpCauchy}).
\item  Unbounded: $f_0(x)$ is  $\Gamma(0.4, 1)$ (Figure \ref{fig:GammaUnbounded}).
\item Mixture with different $z$: $f_0(x)$ is  $0.5 \Gamma(1,3) + 0.5 \Gamma(2,10)$ (Figure \ref{fig:MixtureDiffZ}).
\end{enumerate}
Even though the folded Cauchy density does not satisfy the conditions of the theorem, we show that the proposed gamma mixture still provides a reasonably concentrating posterior distribution.

1000 thinned samples from the posterior distribution are plotted for each true density after at least 50000 burn in iterations, with the red line representing the mean of the posterior distribution and the green line representing the true density. Improvement of the concentration of the posterior distribution with increasing sample size is presented for the two component mixture in Figure~\ref{fig:MixtureDiffZ} (see also Table~\ref{table:L1Norm} for other densities). For all considered true densities, including the unbounded one and the folded Cauchy density, the proposed gamma mixture model performs well. However, the value of the folded Cauchy density around 0 has high uncertainty. Sensitivity with respect to the choice of the free parameters was investigated for all densities, all leading to good performance (presented for the unbounded density $Gamma(0.4, 1)$). We found that using $m=10$ leads to a high number of mixture components even in the cases $\Gamma(0.4,1)$ and $Exp(1)$ (there is only one component if $m=1$ or $m=0.1$).

We study sensitivity of the quality of estimation with regard to the considered loss function, the $L_1$ norm $||f-f_0||_1$, with respect to the DP mass parameter $m$ and  the sample size $n$. The median and the 95\% quantile of the posterior distribution of  $||f-f_0||_1$ for the considered densities with different values of $m$ (0.1, 1, 10) and different sample sizes for the default value $m=1$ ($n=100$ and $n=1000$) are presented in Table~\ref{table:L1Norm}. The quantiles decrease with increasing sample size, and are little affected by changing the value of $m$.

\begin{table}
\begin{center}
\begin{tabular}{|l| l|l|   c| c| }
\hline
Distribution & $m$ & $n$  &50\%&95\%\\
\hline
Cauchy & $m=1$ & $n=100$   &0.187&0.303\\
Cauchy& $m=1$& $n=1000$   &0.071&0.109\\
Cauchy& $m=0.1$& $n=1000$  &0.073&0.108\\
Cauchy& $m=10$ & $n=1000$  &0.065&0.100\\
\hline
Exp(1) & $m=1$& $n=100$  &0.0933&0.1873\\
Exp(1) & $m=1$& $n=1000$  &0.0323 &0.0605\\
Exp(1) & $m=0.1$& $n=1000$   &0.0312 &0.0609\\
Exp(1) & $m=10$& $n=1000$   &0.0383&0.0727\\
\hline
Gamma mixture& $m=1$& $n=100$  &0.1527 &0.2549\\
Gamma mixture& $m=1$& $n=1000$  &0.0944 &0.1170\\
Gamma mixture& $m=0.1$& $n=1000$  &0.0931 &0.1184\\
Gamma mixture& $m=10$& $n=1000$  &0.0692&0.0961\\
\hline
Gamma(0.4,1)& $m=1$& $n=100$ & 0.1274 & 0.2267\\
Gamma(0.4,1)& $m=1$& $n=1000$ & 0.0400&0.0631\\
Gamma(0.4,1)& $m=0.1$& $n=1000$  &0.0321 &0.0655\\
Gamma(0.4,1) & $m=10$ & $n=1000$  & 0.0347 &0.0658\\
\hline
\end{tabular}
\end{center}
\caption{\textsl{ 50\% and 95\% quantiles of the posterior distribution of $||f-f_0||_1$, for different values of the Dirichlet Process prior mass $m$ and sample size $n$.} }\label{table:L1Norm}
\end{table}

\subsection{Email arrival data}

\begin{figure}
 \centering
 \includegraphics[width=4.6cm ]{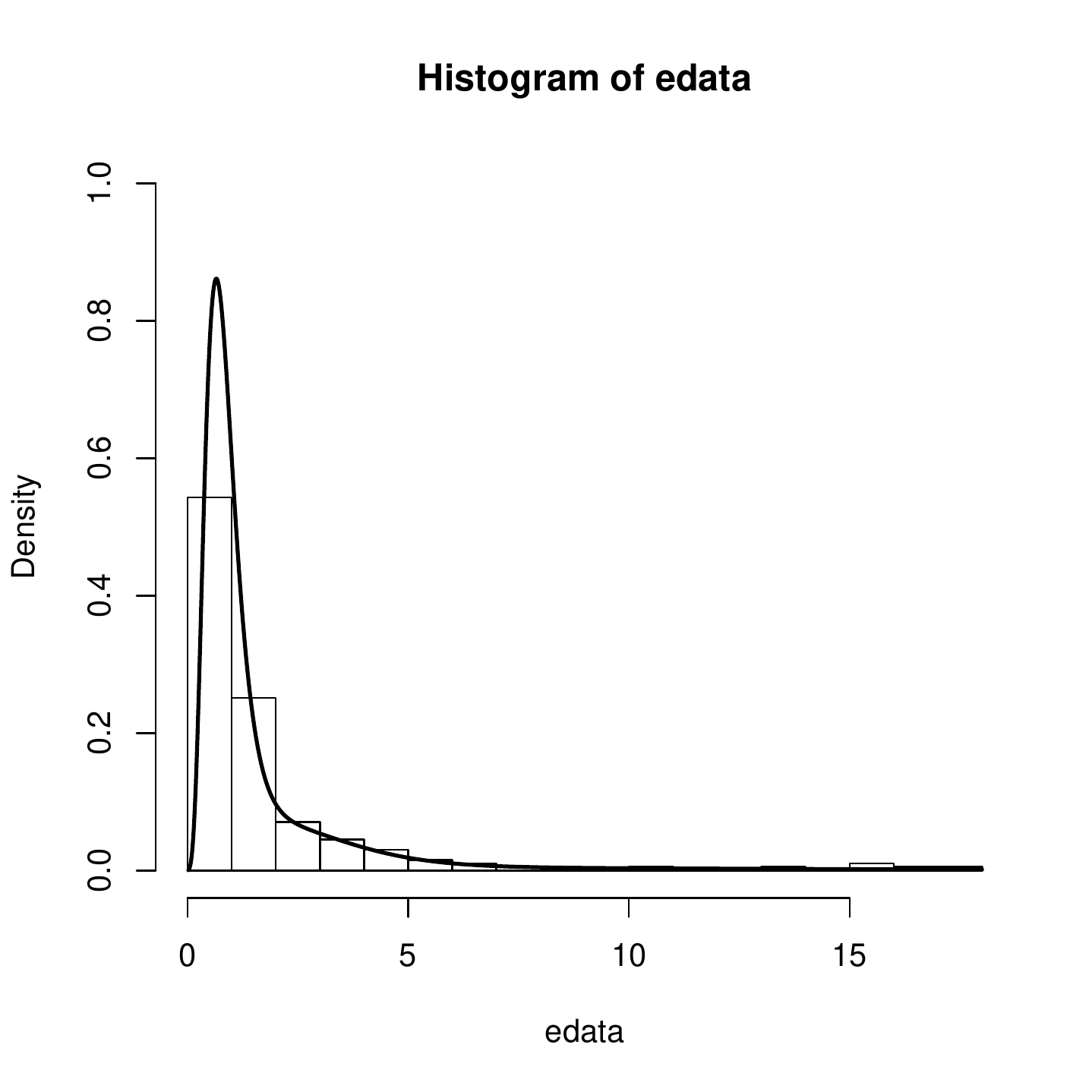}
 \includegraphics[width=4.6cm ]{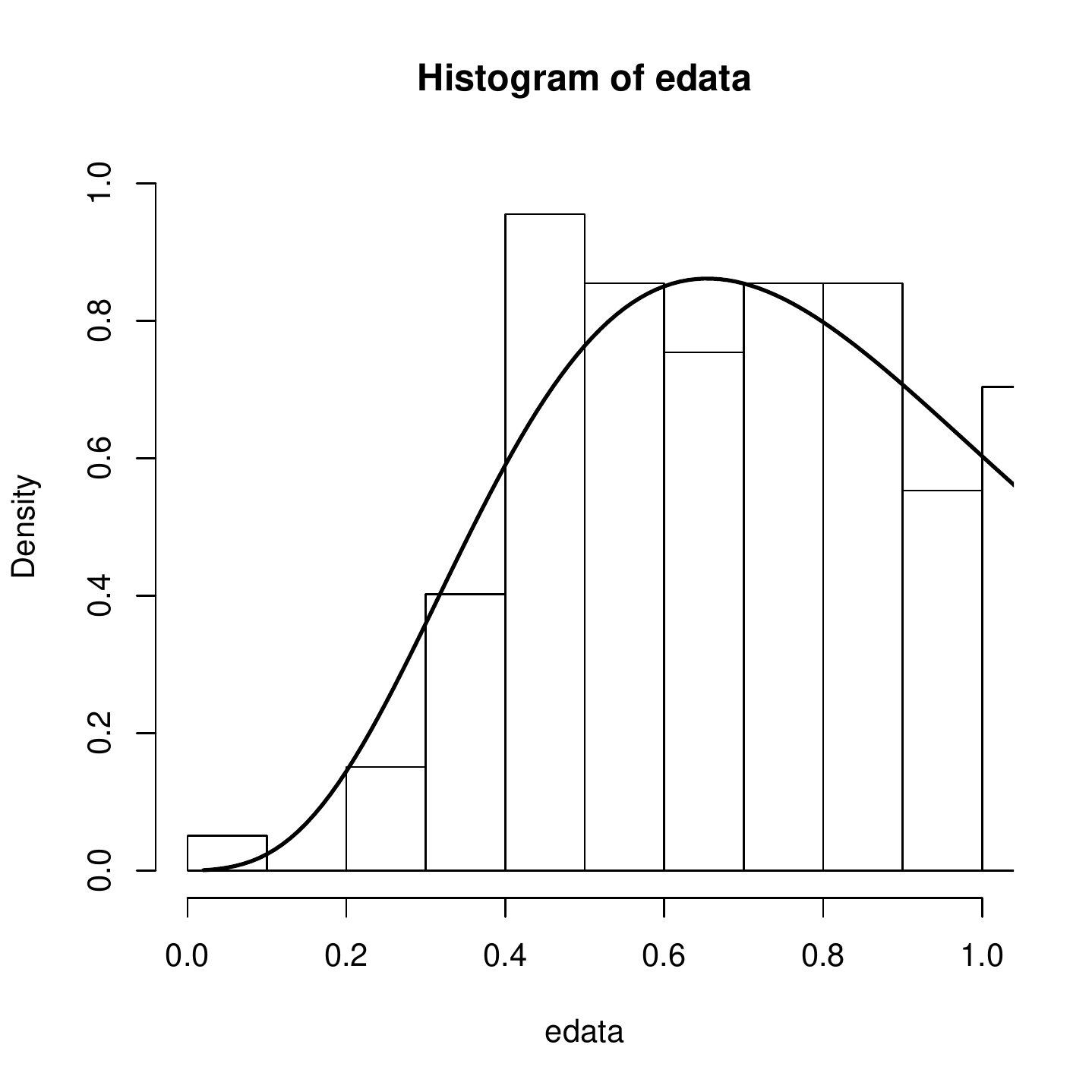}
 \includegraphics[width=4.6cm ]{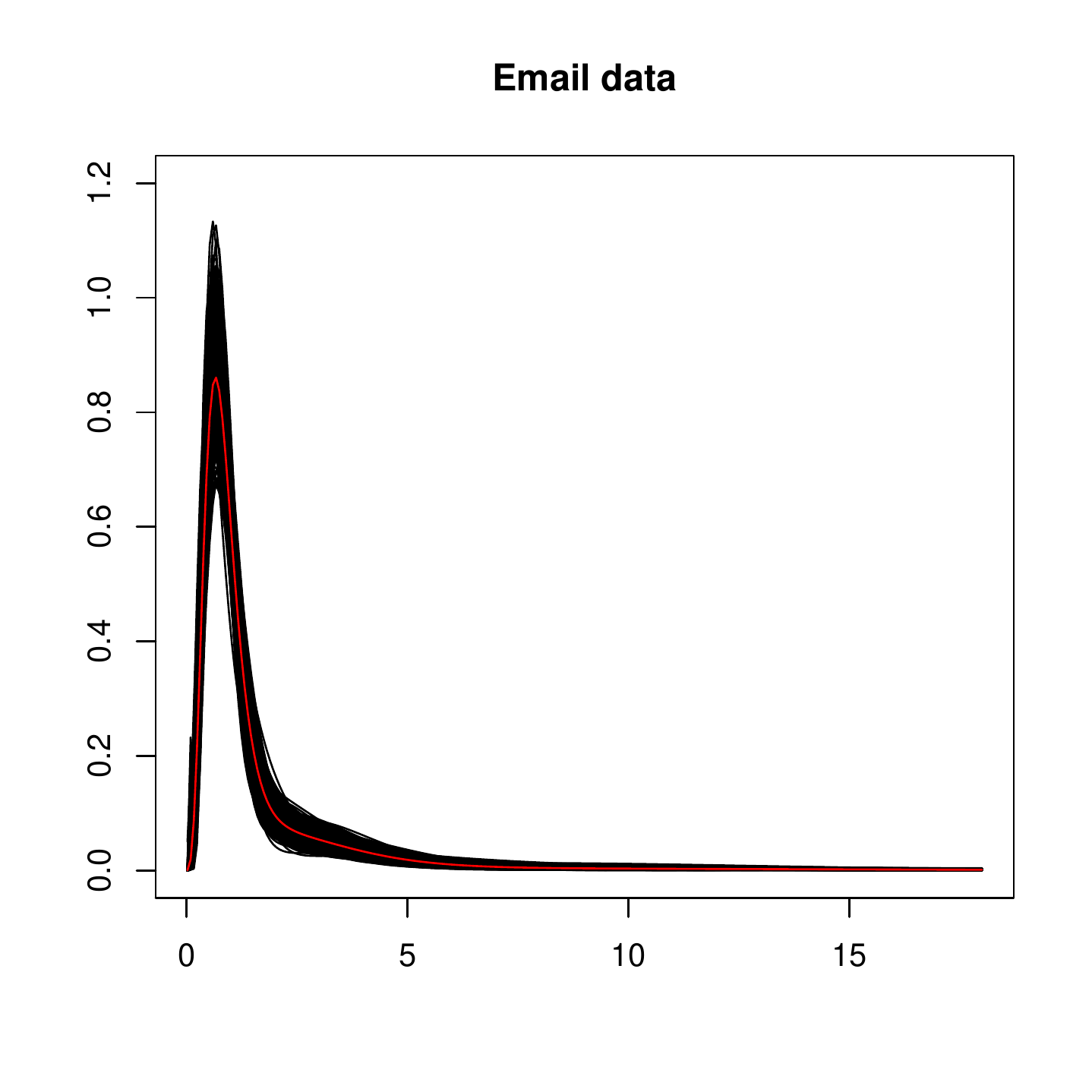}
\caption{ Email arrival data. Top left: histogram of the data with the posterior mean density estimate (black line).  Top right: zoom in of the histogram of the data with the posterior mean density estimate (black line). Bottom:  draws from the posterior distribution with posterior mean (red line). }\label{fig:EmailData}
 \end{figure}

In this section we consider the data of the intervals between arrival times of emails modelled in \cite{GammaMixturesRuggeri2001} which consists of the interarrival times (minutes) of 203 E-mail messages (we are grateful to Fabrizio Ruggeri, one of  the authors, who has kindly provided the data). The proposed location Gamma mixture  \eqref{eq:DefaultPrior} with the default choice of the free parameters was fitted to the data. 1250 samples from the posterior were used which are the last 25\% of the  50000 iterations thinned by 10. The histogram of the data with superimposed posterior mean and the samples from the posterior distribution are shown in Figure~\ref{fig:EmailData}. The histogram shows that the fit of the location mixture is similar to the fit of the location-scale mixture presented in \cite{GammaMixturesRuggeri2001}. The plot of the  samples with the posterior indicates uncertainty about the values of the density around 0 as well as high uncertainty about the possible second mode around 3. We also present a zoom in into the neighbourhood of 0 which confirms the findings that the density is small around 0 and that the distribution of the email arrival times differs from exponential.

\section{Proofs}\label{sec:Proofs}
\subsection{Proof of Theorem \ref{th:rate}} \label{sec:pr:thrate}

\begin{proof}[Proof of Theorem \ref{th:rate}]
The proof consists in verifying the assumptions of Theorem 2.1 of \citet{ghosal:ghosh:vdv:00}.
The first assumption, on the prior mass of Kullback - Leibler neighbourhood of the true density, is verified in Lemma~\ref{lem:DiscrApproxBoundary}.
Note that there is a small modification from condition (2.4) of Theorem 2.1 of \citet{ghosal:ghosh:vdv:00} in that the  bound on the variance of the log-likelihood ratio has an extra $\log n$ term. This does not affect the conclusion since the variance term only need to be negligible compared to $n\KL( f , K_z*P)^2$.
In Lemma~\ref{prop:Entropy} we control the $L_1$ (and Hellinger) entropy of the sieves which are defined below.

Fix an arbitrary $\zeta >0$ to be defined later, and take a sieve $Q_n = Q(\zeta\varepsilon_n, J_n , a_n, b_n,\underline{z}, \bar{z}) $ as defined by  \eqref{def:sieveQ} in Lemma~\ref{prop:Entropy} with
\begin{gather*}
 \varepsilon_n =  n^{-\beta/(2\beta+ 1)} [\log n]^{(5\beta+1)/(4\beta+2)},  \quad  J_n = J_0 n^{1/(2\beta+1)} [\log n]^{3\beta/(2\beta+1)},\\
 a_n = \exp(- C(n(\log n)^{5\beta+1})^{ 1/(2\beta+1)}), \quad  b_n =  \exp( C(n(\log n)^{5\beta+1})^{ 1/(2\beta+1)}),  \\
  \underline z =  \exp(- z_0(n(\log n)^{5\beta+1})^{ 1/(2\beta+1)} ), \quad \bar z =  z_0 n^{2/(2\beta+1)} (\log n)^{2(q - \rho_z)}
\end{gather*}
for some constants $C, z_0 $ and $J_0$ large enough and $q$ as defined in the theorem.

Lemma \ref{prop:Entropy} implies that we need to verify whether these constants satisfy the following conditions:
 \begin{eqnarray}\label{eq:EntropyConditions}
\log N( \zeta \epsilon_n, Q_n, \| \cdot \|_1) \leq  \tilde c n \epsilon_n^2,\quad
\Pi(Q_n^c) \lesssim e^{-\tilde c n \epsilon_n^2}
\end{eqnarray}
 for any $\tilde c >0$ with $\epsilon_n=\epsilon_0 \varepsilon_n$ by choosing $\epsilon_0$ large enough. This corresponds to choosing $\varepsilon = \zeta \epsilon_n/5$ in Lemma \ref{prop:Entropy}.

  The second inequality in \eqref{eq:EntropyConditions} holds if
\begin{eqnarray*}
J_n \bar{G}((0,a_n)) &\lesssim& 
e^{-c n \varepsilon_n^2},\\
J_n \bar{G}((a_n+b_n, \infty)) &\lesssim& 
e^{-c n \varepsilon_n^2},\\
1 - H([\underline{z}, \bar{z}]) &\lesssim& e^{-c n \varepsilon_n^2},\\
\left( \frac{e m}{J_n} \log(1/\varepsilon_n) \right)^{J_n} &\lesssim& e^{-c n \varepsilon_n^2}.
\end{eqnarray*}
In our case, $\varepsilon_n= n^{-\gamma} (\log n)^t$ with $\gamma = \beta/(2\beta+1)$ and $t=(5\beta+1)/(4\beta+2)$. The last condition holds if
$$J_n(\log J_n - \log\log n + C)\geq c n^{1-2\gamma} [\log n]^{2t},
$$
e.g. if $J_n = C n^{1-2\gamma} [\log n]^{2t-1} = C n^{1/(2\beta+1)} [\log n]^{3\beta/(2\beta+1)}$ .

The first inequality in \eqref{eq:EntropyConditions} holds if
$$
J_n  \left[   \log \log(b_n/a_n) +  \log(\bar{z}_n)  + \log (1/\varepsilon_n) \right] +\log\log(\bar{z}_n/\underline{z}_n) \lesssim  n \varepsilon_n^2 = n^{1-2\gamma} [\log n]^{2t}
$$
that is, if $b_n/a_n \lesssim e^{B n^A}$ for some $A, B>0$, $\bar{z}_n\lesssim  n^D$ for some $D>0$,
$$
 n^{1-2\gamma} [\log n]^{2t-1} \left\{  \log n + \log\log n   + C\right\} + \log n \leq  C  n^{1-2\gamma} [\log n]^{2t}
$$
which holds for large enough constant $C$, and if $\log\log(\bar{z}_n/\underline{z}_n) \lesssim n^{1-2\gamma} [\log n]^{2t}$.
 In our case, $b_n/a_n = \exp( C(n(\log n)^{5\beta+1})^{ 1/(2\beta+1)} )\lesssim e^{B n^A}$ with any $A> 1/(2\beta+1)$, and $\bar{z}_n\lesssim  n^D$ with any $D > 2/(2\beta+1)$.

  Choosing $\epsilon_0$ large enough in the definition of $\epsilon_n$ completes the proof of \eqref{eq:EntropyConditions}, and hence the proof of Theorem \ref{th:rate}.

\end{proof}

We extend the definition of $K_z$ in the following way: for any distribution $P$, define
 $$
 K_z * P (x) =   \int_0^\infty g_{z,\epsilon}(x) d P(\epsilon).
 $$
If $P$ has Lebesgue density $f$, then $K_z * P (x)= K_z f(x)$.

\begin{lemma}\label{lem:DiscrApproxBoundary}
Assume that the probability density $f_0 \in \mathcal P_{\alpha}(\beta, L, \gamma,  C_0, C_1, e, \Delta)$   and that there exist $C>0$ and $\rho_1>0$  such that
 $$ \int_{x}^\infty y^2 f_0(y) dy  \leq C (1+x)^{-\rho_1}. $$ Then, for any $\epsilon_0>0$, there exist $ \kappa, C_p >0$ such that
 \begin{equation*}
 \Pi\left( \KL( f , K_z*P) \leq \epsilon_n^2; V( f, K_z*P)\leq \epsilon_n^2 \log n \right) \geq  C_p e^{- \kappa n^{1/(2\beta+1)} (\log n)^{2q} },
 \end{equation*}
 for any prior satisfying condition $(\mathbf{\mathcal P})$ and $n\geq 1$ where $\epsilon_n = \epsilon_0 n^{-\beta/(2\beta+1)} (\log n)^{q}$,  with $q$ defined in Theorem \ref{th:rate}.    The constants $\kappa$ and $C_p $ depend on $\Pi$, $\epsilon_0$  and on the constants defining the functional class.
 \end{lemma}
Lemma \ref{lem:DiscrApproxBoundary} is proved in Section \ref{sec:pr:KLNeigh}.

As in \citet{ghosal:shen:tokdar}, we control the entropy of the following approximating sets.
\begin{lemma}\label{prop:Entropy}
Fix $\varepsilon >0$, $J\in \mathbb{N}$,  $a,b>0$, $0< \underline{z}< \bar{z}<\infty$ and introduce the following class of densities:
\begin{eqnarray}\label{def:sieveQ}
Q = Q(\varepsilon, J , a, b, \underline{z}, \bar{z}) =  \left\{
\begin{array}{ll}
 f = & \sum_{j=1}^{\infty} \pi_j g_{z,\epsilon_j}: \, \sum_{j>J} \pi_j <\varepsilon,  \,\, z \in [\underline{z}, \bar{z}], \\
  & \epsilon_j \in [a, a+b] \, \text{ for } \, j=1,\ldots,J
\end{array}
 \right\}.
\end{eqnarray}

Then, for $\varepsilon \leq  \sqrt{\bar{z}}$,
\begin{eqnarray*}
\log N(5\varepsilon, Q, \| \cdot \|_1) &\leq& C+ J \left[ \log \log\left(\frac b  a \right)   - 3 \log \varepsilon + 0.5 \log (\bar{z})\right]
  +\log \log\left(\frac {\bar{z}}{\underline{z}}\right),\\
\Pi(Q^c) &\leq&  \left( \frac{e m}{J} \log(1/\varepsilon) \right)^J + J (1 -G([a, a+b])) + 1 - \Pi_z([\underline{z}, \bar{z}]),
\end{eqnarray*}
where $\Pi$ is a prior satisfying condition $({\mathcal P})$.
\end{lemma}
The proof of Lemma \ref{prop:Entropy} is given in Section \ref{sec:pr:propEnt}.
 In the next section we prove Lemma \ref{lem:DiscrApproxBoundary}.

\subsection{Proof of Lemma \ref{lem:DiscrApproxBoundary}} \label{sec:pr:KLNeigh}

Consider $P_N$ the discrete distribution constructed in Lemma \ref{lem:KL:PN}, which we write as
 $P_N  = \sum_{j=1}^{N} p_{j} \delta_{u_{j}}$,  with $N   \leq N_0 \sqrt{z}(\log z)^{3/2}$, $ u_j \in [e_z, E_z]$, $u_1\leq u_2 \leq \cdots, \leq u_N$, $u_{i+1} - u_i > z^{-A}$ and $p_j > z^{-A}$ for some $A$ and with $e_z = z^{- a} $ and $E_z = z^b$, with $a,b$ defined as in Lemma \ref{lem:KL:PN}.

 Set $U_{j}  = [(u_{j}+u_{j-1})/2, (u_{j}+u_{j+1})/2]$, with $u_{0} = u_{1}$ and $u_{N+1}= u_{N}$,  $U_0 = \R^+ \setminus \cup_{j=1}^N U_j$ and define for $A>0$
$$\mathcal P_z = \{ P: \, P(U_{j})/p_{j} \in ( 1 - 2 z^{-A}, 1 - z^{-A}) \quad \forall 1 \leq j  \leq N \}. $$
Note that for all $P \in \mathcal P_z$, $P(U_0) \in ( z^{-A}, 2z^{-A})$. 

Let $z_n =n^{2/(2\beta + 1)} (\log n)^{t}$ with $t =2q - 5 $ if $\rho_z \leq 5/2$ and $t = 2(q-\rho_z)$ if $\rho_z > 5/2$ and set  $I_n = (z_n, 2z_n)$. Then for all $z\in I_n$ and all
$P \in \mathcal P_{2z_n}$, Lemma  \ref{lem:KL:discrete2} implies that
\begin{equation*}
\begin{split}
\KL(f, K_z*P)  &\leq n^{-2\beta/(2\beta+1)} (\log n)^{-2\beta t +1} \asymp  \epsilon_n^2 ,\\
 V(f, K_z*P) &\leq C_2n^{-2\beta/(2\beta+1)} (\log n)^{-2\beta t +2} \asymp \epsilon_n^2 \log n,
\end{split}
 \end{equation*}
 if $A$ is large enough (depending on $\beta, L, \gamma,  C_0, C_1, e, \Delta$).

 To prove Lemma \ref{lem:DiscrApproxBoundary}, we thus need to  bound from below $\Pi(I_n\times \mathcal P_n)$.
 Denote $\alpha_{j} = mG(U_{j})$, $j=0, \cdots, N$ with $N \asymp \sqrt{z_n} (\log z_n)^{3/2}  $ in the $DP(m, G)$ type prior case.
 Note that for large $u_{j-1} \gtrsim E_z $,
 \begin{eqnarray*}
 \alpha_{j} &=& m G(U_{j}) =m\int_{(u_{j-1}+u_j)/2}^{(u_j+u_{j+1})/2} g(u) du \geq C  \int_{(u_{j-1}+u_j)/2}^{(u_j+u_{j+1})/2} u^{-a_1} du
\geq C  [ E_z]^{1-a_1}\\
&=& C z^{b(1-a_1)},
  \end{eqnarray*}
and similarly $ \alpha_{j} \leq  C z^{-b(a_1'-1)}$.
For small $u_{j+1} \lesssim  e_z$,
 $$
 \alpha_{j} = m G(U_{j})\geq C \int_{(u_{j-1}+u_j)/2}^{(u_j+u_{j+1})/2} u^{a_0'} du \geq C  [e^z]^{a_0'+1}  =   C z^{-a(a_0'+1)},
  $$
  and similarly $\alpha_j \leq  C z^{-a(a_0+1)}$.
 Hence we have $\alpha_j \geq C z^{-B}$ with $B=\min(a(a_0+1), b(a_1'-1))$.

In particular, we have that $\sum_j\alpha_j=m$ for the DP prior, and
$$
\sum_{j=1}^N (-\log \alpha_j) \leq \sum_{j=1}^N B \log z = N B \log z \leq  N B \log N.
$$
 Also, we have that
$$
\sum_{i=1}^N \alpha_i (\log \alpha_i)_+ =\sum_{i=1}^N m p_i (\log m p_i)_+\leq m\log m 
$$

Note that for $x \in (0,1]$, $\Gamma(x)\leq x^{-1}$, and for $x>1$,
 $\Gamma(x) < \exp(  (x-1/2) \log x -(x-1))$ \cite{GammaFnInequalities}.


Adapting Lemma 6.1 of \citet{ghosal:ghosh:vdv:00} to the case of hyperparameters $\alpha_i$ of the Dirichlet distribution possibly greater than $1$ , we obtain:
\begin{equation*}
\begin{split}
 \Pi( \mathcal P_n) &\geq \frac{  \Gamma(m) }{ \prod_{i} \Gamma( \alpha_{i})}z_n^{-A(\alpha_0-1)}2^{-(\alpha_0-1)_-}
  \int \prod_{i=1}^N\1_{x_i \in (p_i(1-2z_n^{-A}), p_i(1 -z_n^{-A}))}x_i^{\alpha_i-1}d x_i \\
  &\gtrsim z_n^{-A(\alpha_0-1+N)} \frac{  \Gamma(m) }{ \prod_{i} \Gamma( \alpha_{i})} \prod_{i=1}^N p_i^{\alpha_i}  \\
&  \gtrsim  z_n^{-A( N + m   + B N \log N) } \gtrsim z_n^{-(B+1)A N \log N } \gtrsim e^{- (B+2)AN_0 \sqrt{z_n} (\log z_n)^{5/2}}
\end{split}
\end{equation*}
 Using condition $(\mathcal P)$ on $\Pi_z$ we have that
$\Pi_z(I_n) \gtrsim e^{-c \sqrt{z_n} (\log z_n)^{\rho_z}}$, replacing $z_n$ by its expression terminates the proof of Lemma~\ref{lem:DiscrApproxBoundary} for the DP prior.

For the mixture prior satisfying $(\mathcal P)$, we have
\begin{equation*}
\begin{split}
 \Pi( \mathcal P_n) &\geq \frac{  \Gamma(\sum_i \alpha_i) }{ \prod_{i} \Gamma( \alpha_{i})}z_n^{-A(\alpha_0-1)}2^{-(\alpha_0-1)_-}
  \int \prod_{i=1}^N\1_{x_i \in (p_i(1-2z_n^{-A}), p_i(1 -z_n^{-A}))}x_i^{\alpha_i-1}d x_i \\
  &\gtrsim z_n^{-A(\alpha_0-1+N)} \frac{  \Gamma(\sum_i \alpha_i) }{ \prod_{i} \Gamma( \alpha_{i})} \prod_{i=1}^N p_i^{\alpha_i}  \quad
 \gtrsim  B^N z_n^{-A(\alpha_0-1+N)}\\
  & \times \exp\left(- A\log z_n\sum_{i: \alpha_i> 1} \alpha_i+ \sum_i (-\log \alpha_i)_+ - \sum_{i: \alpha_i>1} (\alpha_i-1/2)\log \alpha_i \right)\\
&\gtrsim \exp\left( -  C(N + N\log n) \right),
\end{split}
\end{equation*}
which terminates the proof.


\subsection{Proof of Lemma \ref{prop:Entropy}} \label{sec:pr:propEnt}


Take any $f\in Q$, that is, $f = \sum_{j=1}^{\infty} \pi_j g_{z,\epsilon_j}$ such that $\sum_{j>J} \pi_j <\varepsilon$, $z \in [\underline{z}, \bar{z}]$ and $\epsilon_j \in [a, a+b]$  for $j=1,\ldots,J$.

Fix $\delta_2 = \varepsilon/C$ for the constant $C$ to be defined below, and $\delta_0 =\varepsilon/\sqrt{2z}$.
Let $\hat{A}$ be the following set $\{ a (1+\delta)^k\}_{k=0}^{K}$ with $ K =K_z= \lceil \log(1+b/a)/\log(1+\delta)\rceil$, and $\hat{Z} =\{  \underline{z}(1+\delta_2)^\ell\}_{\ell=0}^L$ with  $ L = \lceil \log(\bar{z}/\underline{z})/\log(1+\delta_2)\rceil$.
 In particular, for any $z \in [\underline{z}(1+\delta_2)^\ell, \underline{z}(1+\delta_2)^{\ell+1})$ for some $\ell \in \{0,1,\ldots, L\}$, $\inf_{\hat{z} \in \hat{Z}} |\hat{z}/z-1| \leq \delta_2$. Let $\hat{S}$ be an $\varepsilon$-net for $S = \{ (p_1,\ldots, p_J): \, p_j = \pi_j/(1-\sum_{j=1}^J \pi_j)  \, \forall j \}$.

Define
\begin{eqnarray*}
\hat{Q} = \left\{
\begin{array}{l}
\hat{f}= \sum_{j=1}^{J} \hat\pi_j g_{\hat{z},\hat\epsilon_j} \text{ where } \hat{z}\in \hat{Z},  \quad  |\hat{z}/z -1| < \delta_2, \quad \\
\quad  \hat\epsilon_j \in \hat{A}, \, j=1,\ldots,J, \,\,\max_{j=1,\ldots,J}|\hat\epsilon_j/\epsilon_j-1| < \delta =\varepsilon/\sqrt{2z},\\
\quad  \hat\pi=(\hat\pi_j) \in \hat{S} \text{ and } \sum_{j=1}^{J} |\hat\pi_j -\tilde\pi_j| < \varepsilon, \quad \text{with }  \tilde\pi_j = \pi_j/[\sum_{j=1}^J \pi_j]
\end{array}
 \right\}.
\end{eqnarray*}

Now we show that $\hat{Q}$ is a $5\varepsilon$-net of $Q$:
\begin{gather*}
||\sum_{j=1}^{J} \hat\pi_j g_{\hat{z},\hat\epsilon_j} - \sum_{j=1}^{\infty} \pi_j g_{z,\epsilon_j}||_1 \leq
||\sum_{j=1}^{J} \hat\pi_j g_{z,\hat\epsilon_j} - \sum_{j=1}^{J} \hat\pi_j g_{\hat{z},\hat\epsilon_j}||_1 +
||\sum_{j>J} \pi_j g_{z,\epsilon_j}||_1\\
+ ||\sum_{j=1}^{J}  \pi_j (g_{z,\hat\epsilon_j}-g_{z,\epsilon_j})||_1+
||\sum_{j=1}^{J} (\hat\pi_j - \pi_j) g_{z,\hat\epsilon_j}||_1 \\
\leq
\sum_{j=1}^{J} \hat\pi_j ||g_{z,\hat\epsilon_j} -  g_{\hat{z},\hat\epsilon_j}||_1 +
\sum_{j>J} \pi_j  + \sum_{j=1}^{J}  \pi_j ||g_{z,\hat\epsilon_j}-g_{z,\epsilon_j}||_1+
 \sum_{j=1}^{J} |\hat\pi_j - \pi_j|.
\end{gather*}
The second term is less than $\varepsilon$ by the definition of $Q$. The last term is bounded in the same way as in \cite{ghosal:shen:tokdar} by
$\sum_{j=1}^{J} |\hat\pi_j - \pi_j| \leq 2\varepsilon$.

 To bound the third term, we first bound the $L_1$ distance between the two gamma densities using Lemma \ref{L1:geps}:
\begin{eqnarray*}
 ||g_{z,\hat\epsilon_j}-g_{z,\epsilon_j}||_1 &\leq& \sqrt{2\KL(g_{z,\hat\epsilon_j}, g_{z,\epsilon_j})} \leq \sqrt{2 z} \delta = \varepsilon
\end{eqnarray*}
by the definition of $\delta$.

To bound the first term, we bound the  Kullback-Leibler distance between the corresponding probability distributions:
\begin{eqnarray*}
\KL(g_{z,\epsilon},  g_{\hat{z},\epsilon}) &=&   \log\left( \frac{\hat{z}^{-\hat{z}}\Gamma(\hat{z})}{z^{-z} \Gamma(z)} \right) -(\hat{z}-z) \left[    \Gamma'(z)/\Gamma(z) - \log z -1\right] \\
&=& 0.5 (\hat{z}-z)^2 [-z_c^{-1} + \psi_1(z_c)]
\end{eqnarray*}
for some $z_c$ between $z$ and  $\hat{z}$ where $\psi_1(z) = (\log \Gamma(z))''$ is the trigamma function. It is known that as $z\to 0$, $\psi_1(z) = \gamma+z^{-2} +o(1)$, and as $z\to \infty$, $\psi_1(z) = z^{-1}+ 0.5z^{-2} +o(z^{-2})$ which implies that for both $z$ large and small,
$$
\psi_1(z) -1/z = O(z^{-2}),
$$
which implies that we can bound the  Kullback-Leibler distance as
\begin{eqnarray*}
\KL(g_{z,\epsilon},  g_{\hat{z},\epsilon}) = 0.5 (\hat{z}-z)^2 [-z_c^{-1} + \psi_1(z_c)]\leq C (\hat{z}/z-1)^2 \leq C \delta_2^2
\end{eqnarray*}
for an appropriate constant $C$. Therefore, the first term is bounded by
\begin{eqnarray*}
\sum_{j=1}^{J} \hat\pi_j ||g_{z,\hat\epsilon_j} -  g_{\hat{z},\hat\epsilon_j}||_1 &\leq& \sum_{j=1}^{J} \hat\pi_j \sqrt{2 \KL(g_{z,\hat\epsilon_j},  g_{\hat{z},\hat\epsilon_j})} \leq  \sqrt{2C}\delta_2 = \varepsilon
\end{eqnarray*}
by the definition of $\delta_2$.


Now we study cardinality of set $\hat{A}$. For each  $z \in [\underline{z}, \bar{z}]$ ,
\begin{eqnarray*}
 K_z &\leq& 1 +   \log(1+b/a)/\log(1+\varepsilon/\sqrt{2z})
 \lesssim \frac{\log(b/a)}{\log(1+\varepsilon/\sqrt{2\bar{z}})} \lesssim \frac{\sqrt{\bar{z}} \log(b/a)}{ \varepsilon}
  \end{eqnarray*}
for large $b/a$ to due to $\log(1+x) \geq x (1-0.5/\sqrt{2})$ for $x\leq 1/\sqrt{2}$ and assuming that $\varepsilon \leq  \sqrt{\bar{z}}$.

Cardinality of $\hat{S}$ is $\lesssim \varepsilon^{-J}$  \citep[proof of Proposition 2]{ghosal:shen:tokdar}.

Then, for $\varepsilon \leq  \sqrt{\bar{z}}$, the cardinality of $\hat{Q}$ is bounded by
\begin{eqnarray*}
|\hat{Q}| &\leq& |\hat{S}| \sum_{\ell=1}^L |K_z|^J
\lesssim  \varepsilon^{-J} L \left[ \frac{\sqrt{\bar{z}} \log(b/a)}{ \varepsilon}\right]^{J}
\lesssim  \left[\frac{\log(b/a)\sqrt{\bar{z}}}{\varepsilon^2} \right]^{J} \frac{\log(\bar{z}/\underline{z})}{ \varepsilon}
\end{eqnarray*}
  due to $\delta_2 = C \varepsilon$ and by the definition of $L$.

Therefore, combining all the inequalities together, we obtain that
\begin{eqnarray*}
||\sum_{j=1}^{J} \hat\pi_j g_{\hat{z},\hat\epsilon_j} - \sum_{j=1}^{\infty} \pi_j g_{z,\epsilon_j}||_1 &\leq& 5 \varepsilon,
\end{eqnarray*}
and hence $\hat{Q}$ is a $5\varepsilon$-net of $Q$, with
\begin{eqnarray*}
\log N(5\varepsilon, Q, ||\cdot||_1) &\leq& C+ J \left[ \log \log(b/a)   - 2 \log \varepsilon + 0.5 \log (\bar{z})\right]
  +\log \log(\bar{z}/\underline{z})-\log \varepsilon\\
   &\leq& C+ J \left[ \log \log(b/a)   - 3 \log \varepsilon + 0.5 \log (\bar{z})\right]
  +\log \log(\bar{z}/\underline{z}).
\end{eqnarray*}

The second inequality is proved following the same route as in the proof of Proposition 2 in \citet{ghosal:shen:tokdar}.

For the Dirichlet process prior,
\begin{eqnarray*}
\Pi(Q^c) &\leq& \Pi_z([\underline{z}, \bar{z}]^c) + J G( [a, a+b]^c ) +  \Pi(\sum_{j>J} \pi_j >\varepsilon)\\
&\leq& J (1-G( [a, a+b] )) + 1- \Pi_z([\underline{z}, \bar{z}])+ \left( \frac{e m}{J} \log(1/\varepsilon) \right)^J.
\end{eqnarray*}
For the mixture prior,
\begin{eqnarray*}
\Pi(Q^c) &\leq&  \Pi_z([\underline{z}, \bar{z}]^c) + \sum_{k=1}^{J} \pi_K(k)  [G( [a, a+b]^c )]^k +\sum_{k=1}^{\infty} \pi_K(k)\Pi(\sum_{j=J+1}^k \pi_j >\varepsilon)]\\
&\leq&  1- \Pi_z([\underline{z}, \bar{z}])+ J G( [a, a+b]^c ) + \sum_{k=J+1}^{\infty} C e^{-c k (\log k)^{\rho_2}}\\
&\leq&  1- \Pi_z([\underline{z}, \bar{z}])+ J [1-G( [a, a+b] )] \\
&& +  (\log J)^{-\rho_2} e^{-c J (\log J)^{\rho_2}} \frac{c^{-1} C e^{-c (\log J)^{\rho_2}}}{1-e^{-c (\log J)^{\rho_2}}}\\
&\lesssim&  1- \Pi_z([\underline{z}, \bar{z}])+ J [1-G( [a, a+b] )] +  (\log J)^{-\rho_2} e^{-c J (\log J)^{\rho_2}}
\end{eqnarray*}
 since $\rho_2\geq 0$.


\appendix

\section{Proof of Proposition~\ref{lem:approx:1} and related lemmas} 

\subsection{Proof of Proposition~\ref{lem:approx:1}}\label{sec:proof:approx1}

The proof of the proposition is based on the ideas of \citet{rousseau:09}, \citet{kruijer:rousseau:vdv:10} and \citet{ghosal:shen:tokdar}.
First we prove \eqref{pointwise:approx} for $f \in \mathcal P ( \beta, L(\cdot), \gamma, C_0, C_1, e, \Delta) $, and then adapt it for the case $f \in \mathcal P_\alpha ( \beta, L(\cdot), \gamma, C_0, C_1, e, \Delta)$. The proof of \eqref{Hell:contdens} is then presented directly for $f \in  \mathcal P_\alpha ( \beta, L(\cdot), \gamma, C_0, C_1, e, \Delta)$.

\noindent
\underline{1.  Proof of \eqref{pointwise:approx} : Case $\alpha = 1$ }

This case corresponds to $f \in \mathcal P ( \beta, L(\cdot), \gamma, C_0, C_1, e, \Delta) $. We can thus write
$$f(\epsilon)  =  \sum_{j=0}^r \frac{f^{(j)}(x)}{ j !}  ( \epsilon - x)^j + R_1(\epsilon, x)$$
where $|R_1(\epsilon, x)| \leq L(x) |\epsilon - x|^\beta $.
Then, using \eqref{eq:Kzf},
\begin{equation*}
\begin{split}
K_z f(x) &= f(x) I_0(z)+  \sum_{j =1}^r  \frac{  x^jf^{(j)}(x) }{ j ! z^{j/2} }  \mu_j(z) + \frac{R_z(x)}{z^{\beta /2} } ,\\ \quad |R_z(x)| \lesssim   L(x)  x^{\beta }  \left( 1 + \frac{ x^{\gamma} }{ z^{\gamma/2}} \right),
\end{split}
\end{equation*}
where $I_0(z)$ is defined by \eqref{I0} and $\mu_k(z)$ are defined by \eqref{Ik}.
We then have
\begin{equation*}
\Delta_z f(x)  := K_zf(x) - f(x) = \frac{ f(x) }{ z-1 } +   \sum_{j =1}^r  \frac{  x^jf^{(j)}(x) }{ j ! z^{j/2} }  \mu_j(z) + \frac{R_z(x)}{z^{\beta /2} },
\end{equation*}
so that if  $\beta \leq 2$, and since  $\mu_1(z) = O(z^{-H})$ for any $H>0$, we obtain
\begin{equation*}
\Delta_z f(x)  = \frac{ f(x) }{ z-1} + \frac{R_z(x)}{z^{\beta /2} } + O(z^{-H} xf^{(1)}(x)),
\end{equation*}
where the last term only appears if $\beta >1$ and Proposition \ref{lem:approx:1} is verified for $f \in \mathcal P ( \beta, L(\cdot), \gamma, C_0, C_1, e, \Delta) $.  If $\beta \in (2, 4]$, define
 \begin{equation*}
 f_1(x)  = f(x)  - \frac{ f(x) }{ z-1 } - \frac{ x^2 f^{(2)}(x)\mu_2(z)  }{ 2 z },
 \end{equation*}
 then
\begin{eqnarray*}
K_z f_1(x) - f(x) &= -\frac{ \Delta_z f(x) }{ z-1 } - \frac{ \mu_2(z)  \Delta_z(x^2 f^{(2)}(x))  }{ 2 z } + \frac{R_z(x)}{z^{\beta /2} } + O(z^{-H} xf^{(1)}(x)).
\end{eqnarray*}
Note that if $f \in   \mathcal P ( \beta, L(\cdot), \gamma, C_0, C_1, e, \Delta) $, with $\beta >2$, the function  $x \rightarrow x^2 f^{(2)}(x) $ is $r-2$ times continuously differentiable and its derivatives are given by
 \begin{equation}
 (x^2 f^{(2)}(x) )^{(l)} = x^2 f^{(2+l)} (x) + \sum_{j=1}^{2\wedge l } C_l^j 2 \cdot \ldots \cdot (2-j+1) x^{2-j} f^{(2+l-j)}(x)
 \end{equation}
 so that
  $$
  \int_0^\infty\frac{  ( x^l( x^2 f^{(2)}(x) )^{(l)})^2 }{ f(x) } dx \lesssim \sum_{j=2}^r \int_0^\infty\frac{  ( x^j f^{(j)}(x))^2 }{ f(x) } dx < +\infty.
  $$
Hence,
\begin{equation*}
\begin{split}
 & \left| (x^2 f^{(2)})^{(r-2)}(x+y) - (x^2 f^{(2)})^{(r-2)}(x) \right| \\
  &\lesssim \left| (x+y)^2 f^{( r )} (x+y)- x^2 f^{( r )} (x) \right| \\
  & +
 \sum_{l=1}^{2 \wedge (r-2)} \left| (x+y)^{2-l} f^{( r-l )} (x+y)- x^{2-l} f^{( r-l )} (x) \right| \\
 & \lesssim |y| ( |y| +  x ) |f^{ ( r) }(x) | +  ( y^2 + x^2 +1) |y|^{\beta - r+2} L(x) ( 1 + |y|^\gamma)\\
 &+ (|y| + x)|y| |f^{(3)}(x) |
\end{split}
\end{equation*}
where the last term only appears if $r = 3$. Combined with \eqref{I0} and \eqref{Ik}, this leads to a remaining term in the control of $\Delta_z (x^2 f^{(2)}(x))$ bounded by
\begin{eqnarray*}
R(x) &=& Cz^{-(r-1)/2} \left[ x^{r}|f^{ ( r) }(x) | + (x^2+1) x^\beta L(x) ( 1 +  x^\gamma) + x^3 |f^{(3)}(x) |\right]\\
 &:=& z^{-(r-1)/2} \tilde R(x)
\end{eqnarray*}
with $\tilde R$ satisfying
$$\int_0^\infty \frac{\tilde R (x)^2 }{ f(x) } dx < +\infty .$$
$\tilde R$ thus behaves like $R_z(x) $ and we can write
\begin{equation*}
\begin{split}
K_z f_1(x) - f(x) &= -\frac{f(x) }{ (z-1)^2 } - \frac{ \mu_2(z)  x^2 f^{(2)}(x)  }{ z(z-1) } + z^{-\beta/2}R_{2,z}(x)
\end{split}
\end{equation*}
with $$\int_0^\infty \frac{ R_{2,z}(x)^2 }{ f(x) }dx < +\infty$$
uniformly in $z$.
 We can reiterate if $\beta > 4$. At the $k-1$ - th iteration
 $$f_{k-1}(x) = \sum_{j=0}^{2k-2} d_{ k-1,j }(z) \frac{ x^j f^{(j)}(x)  }{ z^{j/2} } $$
 with $d_{k-1,0} = (z/(z-1))^{k-1} $ and for each $j$
  $$\epsilon^j f^{(j)}(\epsilon)  = \sum_{l=0}^{r-j} \frac{ (\epsilon - x)^l }{ l ! } \left\{ \sum_{t=0}^{l \wedge j } x^{j-t} f^{(j+l-t)}(x) a_{j,t} \right\} + \tilde R(\epsilon, x) $$
  so that we can write
  \begin{equation*}
 \begin{split}
 K_z f_{k-1}(x) - f(x) &= \frac{d_{k-1,0}  f(x) }{ z-1}\\
  &+ \sum_{l=1}^{2k-2} \frac{ x^l f^{(l)}(x)  }{ z^{l/2} }
 \left(\frac{ d_{k-1,l} }{ z-1}  + \sum_{j=0}^{l-1}  d_{k-1,j} \sum_{l'=1}^{l-j} \sum_{t =0}^{l'\wedge j} \frac{ a_{j,t} \mu_{l'}(z)}{ z^{t/2}}   \right) \\
  &
   + \frac{ x^{2k} f^{(2k)}(x)  }{ z^{k/2} } \sum_{j=0}^{2k-1}  d_{k-1,j} \sum_{l'=1}^{2k-j} \sum_{t =0}^{l'\wedge j} \frac{ a_{j,t} \mu_{l'}(z)}{ z^{t/2}}\\
 &+ z^{-\beta/2} \tilde R_{k,z}(x)
 \end{split}
 \end{equation*}
 with
 $$\int_0^\infty \frac{ \tilde R_{k,z}(x)^2 }{ f(x) } dx < +\infty $$
 and we define
  \begin{eqnarray*}
  f_k(x) &=& \sum_{l=0}^{2k-2} \frac{ x^l f^{(l)}(x)  }{ z^{l/2} } \left( d_{k-1,l}\frac{z}{ z-1}  -   \sum_{j=0}^{l-1}  d_{k-1,j} \sum_{l'=1}^{l-j} \sum_{t =0}^{l'\wedge j} \frac{ a_{j,t} \mu_{l'}(z)}{ z^{t/2}}   \right) \\
  && -  \frac{ x^{2k} f^{(2k)}(x)  }{ z^{k/2} } \sum_{j=0}^{2k-1}  d_{k-1,j} \sum_{l'=1}^{2k-j} \sum_{t =0}^{l'\wedge j} \frac{ a_{j,t} \mu_{l'}(z)}{ z^{t/2}}
  \end{eqnarray*}
 which corresponds to $f_{k-1} - \Delta_z f_{k-1}$ without the terms $\tilde R_{k,z}(x)$. The recursive relation is
  $$ d_{k,l} = d_{k-1,l}\frac{z}{ z-1}  -   \sum_{j=0}^{l-1}  d_{k-1,j} \sum_{l'=1}^{l-j} \sum_{t =0}^{l'\wedge j} \frac{ a_{j,t} \mu_{l'}(z)}{ z^{t/2}} $$
for $l=0, \cdots , 2k$ with the convention that $d_{k-1, 2k}= 0$.
By construction when $\beta \in (2k-2, 2k]$
 $$\int_0^\infty \frac{ \left(K_z f_{k}(x) - f(x)\right)^2 }{ f(x) } dx  \lesssim z^{-\beta},  $$
that is, we iterate until $k=r_0$.
 Since $\int_0^1 f(x) dx = 1$ and since
 $$\| K_z f_k - f\|_1  \leq \sqrt{\|f\|_1 \int_0^\infty \frac{ \left(K_z f_{k}(x) - f(x)\right)^2 }{ f(x) } dx} \lesssim z^{-\beta/2},$$
 we have that
  $$\int K_z f_k (x) dx = 1 + O(z^{-\beta/2}) \quad \Rightarrow \int_0^\infty f_k(x) dx = 1 + O(z^{-\beta/2}),$$
which proves \eqref{pointwise:approx} for $\tilde f_{\beta, \alpha } = f_{\beta, \alpha} = f_k$.

2.  \underline{Proof of \eqref{pointwise:approx} : Case $\alpha \leq 1$ }

Now let $f\in \mathcal P_{\alpha}( \beta, L(\cdot), \gamma, C_0, C_1, e, \Delta)$ and denote $h(x) = x^{1-\alpha} f(x)$.
Recall that  $C_z = (z -\alpha+1)/z$ and  $\tilde{f}(x) =C_zf(C_zx)$, so $\tilde{f}$ is still a density.
 Note that $C_z = (z -\alpha+1)/z \to 1$ as $z\to \infty$. By  Lemma~\ref{lem:ApproxAlpha},
\begin{eqnarray*}
K_z \tilde{f}(x) &=& x^{\alpha-1} C_{z}^{\alpha-1}  \frac{z^{\alpha}\Gamma(z-\alpha)}{\Gamma(z)} K_{z+1-\alpha} h(x)
\end{eqnarray*}
and $\frac{z^{\alpha}\Gamma(z-\alpha)}{\Gamma(z)} = 1+ O(1/z)$ so we can write
$C_{z}^{\alpha-1}  \frac{z^{\alpha}\Gamma(z-\alpha)}{\Gamma(z)} = 1+ r(z)$,
 where $|r(z)|\leq c/z$ for sufficiently large $z$.

Applying case $\alpha =1$,  for $\mathcal P ( \beta, L(\cdot), \gamma, C_0, C_1, e, \Delta) $ to $h(x)$ with $k$ such that   $\beta \in (2k-2, 2k]$, i.e. $k=r_0$, we have that
 there exists $d_j \in \R$, $j=1, \cdots, k $ such that the function
  $$g_{\beta}(x) =   h(x) - \sum_{j=1}^{2k}\frac{ d_j(z) }{ z^{j/2} } x^j   h^{(j)}(  x), \quad d_j(z) = d_j + O(1/z)  $$
  satisfies
 \begin{gather*}
 \left| K_{z } g_{\beta} (x) -  h(x) \right| \leq z^{-\beta/2} R_{0}(x).
 \end{gather*}
Thus, we can define
 \begin{eqnarray*}
\tilde f_{\beta , \alpha } (x) &=& (1+r(z))^{-1} C_z^{\alpha} x^{\alpha-1}g_{\beta} (xC_z) \\
&=& (1+r(z))^{-1}C_z^{\alpha} x^{\alpha-1}\left[ h(xC_z) - \sum_{j=1}^{r}\frac{ d_{j}(z) }{ z^{j/2} } x^j C_z^j h^{(j)}(xC_z)\right],
\end{eqnarray*}
 which satisfies
 \begin{eqnarray}\label{eq:PointWise1}
 \left| K_{z} \tilde f_{\beta , \alpha } (x) - f(x) \right|  &\leq&  z^{-\beta/2} x^{\alpha-1}R_{0}(x) (1+O(1/z)),
 \end{eqnarray}
 since
 \begin{gather*}
 \left| K_{z} \tilde f_{\beta , \alpha } (x) - f(x) \right|  =   \left|(1+r(z))^{-1} (1+r(z)) x^{\alpha-1} (K_{z+1-\alpha} g_{\beta}(x) - x^{\alpha-1} h(x)) \right| \\
 \leq x^{\alpha-1}  \left| K_{z+1-\alpha} g_{\beta}(x) -  h(x) \right| \leq  z^{-\beta/2} x^{\alpha-1}R_{0}(x) (1+O(1/z)),
 \end{gather*}
 and the first part of \eqref{pointwise:approx} holds with $R(x) \asymp x^{\alpha-1}R_{0}(x)$.
 From the proof of case $\alpha=1$, it follows that $R_0(x)$ has terms proportional to $(x^{2\ell}+1)x^{\beta}(1+x^{\gamma})L_h(x)$ with $1\leq \ell \leq r_0$ and $x^{j} |h^{(j)}(x)|$ for $1\leq j \leq r$. Therefore, the second part of \eqref{pointwise:approx}
 $$
\int_0^{\infty} \frac{R^2(x)}{f(x)} dx \asymp \int_0^{\infty} \frac{R_{0}^2(x)}{h^2(x)} f(x) dx <\infty
$$ is satisfied since
 \begin{eqnarray*}
\int_0^{\infty} \frac{[x^{2\ell}x^{\beta}(1+x^{\gamma})L(x)]^2}{h^2(x)} f(x) dx <\infty,\quad \int_0^{\infty} \frac{[x^{j} |h^{(j)}|]^2}{h^2(x)} f(x)dx <\infty
 \end{eqnarray*}
hold due to $f \in {\mathcal P}_{\alpha}( \beta, L(\cdot), \gamma, C_0, C_1, e, \Delta)$ and inequality $\int g^2(x) d\mu \leq [\int g^p(x) d\mu]^{2/p}$ for $p\geq 2$ for a probability measure $\mu$.

We now prove \eqref{Hell:contdens}.

 3. \underline{Proof of \eqref{Hell:contdens}: general case}
We  follow the same route as in \citet{kruijer:rousseau:vdv:10}. We  bound
\begin{gather}\label{decomp:Hell}
D_H(K_z \bar f_\beta, f)^2 \leq 2[D_H^2(K_z \bar f_\beta, c_\beta f)+D_H^2(f, c_\beta f)]  \leq   2(1-\sqrt{c_\beta})^2 \\
 +  2c_\beta\int  \left(\sqrt{  K_z \left( \tilde f_{\beta,\alpha} \1_{  \tilde f_{\beta,\alpha} \geq \tilde{f}/2} + \frac{ \tilde{f}}{2 }  \1_{ \tilde f_{\beta,\alpha}< \tilde{f}/2} \right)  }-\sqrt{ f(x)}\right)^2 dx. \notag
\end{gather}
We first prove that
  \begin{equation}\label{eq:Cbeta}
   c_\beta^{-1} = \int_0^\infty \left( \tilde f_{\beta,\alpha}(x) \1_{\tilde f_{\beta,\alpha}\geq \tilde f /2} (x) + \tilde f(x)/2 \1_{\tilde f_{\beta,\alpha}< \tilde f /2}(x) \right) dx =1 + O(z^{-\beta/2}).
  \end{equation}
 Define $\tilde{h}(x) = h(C_z x)$ and
    \begin{equation}\label{def:A1}
   \mathcal A_1(a) = \left\{ x:\, \frac{ x^j | \tilde{h}^{(j)}(x)| }{ \tilde{h}(x) } \leq \delta \frac{z^{j/2}}{ (\log z)^{a}}, j=1, \ldots, 2k, \,\, \frac{x^{\beta} L_{\tilde{h}}(x)}{\tilde{h}(x)} \leq  \delta \frac{z^{\beta/2}}{ (\log z)^{a}} \right\},
   \end{equation}
   we have that for $z$ large enough $\{ \tilde f_{\beta , \alpha } < \tilde{f}/2\} \subset \mathcal A_1(a)^c $ (that is, if $\delta \sum_{j=1}^{2k} |d_{k,j}| > (\log z)^{-a}/2$ with $k=r_0$), so that
  \begin{eqnarray*}
    c_\beta^{-1} 
    &\geq& \int  \tilde f_{\beta,\alpha}(x)  dx =   1 + O(z^{-\beta/2}),\\
  c_\beta^{-1}  &=&  1 + O(z^{-\beta/2}) + \int\1_{ \tilde f_{\beta,\alpha}\leq \tilde{f}/2} \left( \frac{\tilde{f} }{ 2} - \tilde f_{\beta,\alpha} \right) dx \\
  &\leq&  (1 + O(z^{-\beta/2}))  + \int\1_{ \tilde f_{\beta,\alpha}\leq \tilde{f}/2} \left( \tilde{f}(x)  + x^{\alpha-1}\sum_{j=1}^{2k} \frac{|d_{k,j}|}{ z^{j/2}} x^j |\tilde{h}^{(j)}|(x) \right) dx.
  \end{eqnarray*}
  Since
  \begin{gather}\label{FEzc2}
    \tilde{F}(\{ \tilde f_{\beta,\alpha}< \tilde{f}/2\}) \leq \tilde{F}(\mathcal A_1(a)^c ) \notag \\
  \leq \sum_{j=1}^{2k} z^{-(2\beta+e)/2} (\log z)^{a(2\beta+e)/j} \int \frac{ (x^{\alpha-1}x^j|\tilde{h}^{(j)}(x)|)^{(2\beta+e)/j} }{ [\tilde{f}(x)]^{(2\beta+e)/j}}
   \tilde{f}(x) dx \notag\\
   = O(z^{-\beta - e/4})
    \end{gather}
  and since for all $j=1, \cdots, 2k$ 
  \begin{eqnarray}\label{FEzcbis2}
  \int_{\mathcal A_1(a)^c} x^{\alpha-1} x^j|\tilde{h}^{(j)}(x)| dx  &\leq& \tilde{F}(\mathcal A_1(a)^c)^{(\beta - j)/\beta} \left[\int \frac{ [x^{\alpha-1} x^j|\tilde{h}^{(j)}(x)|]^{\beta/j} }{ [\tilde{f}(x)]^{\beta/j}}\tilde{f}(x) dx\right]^{j/\beta}\notag\\
 &=& O( \tilde{F}(\mathcal A_1(a)^c)^{(\beta - j)/\beta}) = O( z^{-(\beta - j)(1 + e/(4\beta))} ),
   \end{eqnarray}
which implies \eqref{eq:Cbeta}. We now bound the second term of the right hand side of \eqref{decomp:Hell}:
  \begin{equation*}
  \begin{split}
& \int  \left(\sqrt{ K_z  (\tilde f_{\beta,\alpha} \1_{ \tilde f_{\beta,\alpha} \geq \tilde{f}/2}) + 0.5 K_z (\tilde{f} \1_{ \tilde f_{\beta,\alpha}< \tilde{f}/2})}-\sqrt{ f(x)}\right)^2 dx \\
& \leq  \int \left(\sqrt{  K_z  \tilde f_{\beta,\alpha}(x)}-\sqrt{ f(x)}\right)^2 dx   +
 0.5\int  K_z (\tilde{f} \1_{\mathcal A_1(a)^c})(x)  dx
 \end{split}
  \end{equation*}
since $\{  \tilde f_{\beta,\alpha} < \tilde{f}/2\} \subseteq \mathcal A_1(a)^c$. Using \eqref{pointwise:approx}, we have
\begin{eqnarray*}
 \int \left(\sqrt{  K_z \tilde f_{\beta,\alpha} (x)}-\sqrt{ f(x)}\right)^2 dx &=& \int \frac{ ( K_z \tilde f_{\beta,\alpha} (x)- f(x))^2}{\left(\sqrt{  K_z \tilde f_{\beta,\alpha} (x)}+\sqrt{ f(x)}\right)^2} dx\\
 \leq z^{-\beta}\int_{  \tilde f_{\beta,\alpha} \geq \tilde{f}/2}  \frac{ (x^{\alpha-1}R_0(x))^2}{   f(x)} dx
&=& O(z^{-\beta}).
\end{eqnarray*}


Now we consider the integral $ \int K_z (\tilde{f}\1_{\mathcal A_1(a)^c})(x) dx$.
Note that
\begin{equation*}
\begin{split}
\int K_z (\tilde{f}\1_{\mathcal A_1(a)^c})(x) dx &= \int \int \1_{\mathcal A_1^c(a)}(\epsilon) g_{z,\epsilon}(x) \tilde{f}(\epsilon) d\epsilon dx \\
&\leq  \int \int \1_{\mathcal A_1^c(a)}(\epsilon) \1_{\mathcal A_1(a/2)} (x) g_{z,\epsilon}(x) \tilde{f}(\epsilon) d\epsilon dx + \tilde{F}(\mathcal A_1(a)^c ) \\
&\leq \int \int \1_{\mathcal A_1^c(a/2)}(x) \1_{\mathcal A_1(a)} (\epsilon) g_{z,\epsilon}(x) \tilde{f}(\epsilon) d\epsilon dx+ O(z^{-\beta-e/2})
\end{split}
\end{equation*}
Using \eqref{gaussian:approx}, we have that
\begin{equation*}
\begin{split}
& \int \int \1_{\mathcal A_1^c(a/2)}(x) \1_{\mathcal A_1(a)} (\epsilon) g_{z,\epsilon}(x) \tilde{f}(\epsilon) d\epsilon dx \\
& =\int \int \1_{| \epsilon/ x -1| \leq M \sqrt{\log z}/\sqrt{z}}\1_{\mathcal A_1^c(a/2)}(x) \1_{\mathcal A_1(a)} (\epsilon) g_{z,\epsilon}(x) \tilde{f}(\epsilon) d\epsilon dx + O(z^{-H}),
\end{split}
 \end{equation*}
for any $H>0$ by choosing $M$ large enough since
\begin{equation*}
\begin{split}
& \int \int \1_{| \epsilon/ x -1| > M \sqrt{\log z}/\sqrt{z}}\1_{\mathcal A_1^c(a/2)}(x) \1_{\mathcal A_1(a)} (\epsilon) g_{z,\epsilon}(x) \tilde{f}(\epsilon) d\epsilon dx\\
& \leq \int  \1_{|v| > M \sqrt{\log z}/\sqrt{z}} (v+1)^{-1} \phi_{1/\sqrt{z}}(v)(1+O(1/z)) \int (v+1)\tilde{f}((v+1) x) dx dv \\
& \leq  (1+M \sqrt{\log z}/\sqrt{z})^{-1}\int \1_{|v| > M \sqrt{\log z}/\sqrt{z}} \phi_{1/\sqrt{z}}(v)(1+O(1/z))dv\\
&= 2(1+M \sqrt{\log z}/\sqrt{z})^{-1} (1+O(1/z))[1-\Phi(M \sqrt{\log z}/\sqrt{z})] = O(z^{-H}),
\end{split}
 \end{equation*}
for the appropriate choice of $M$.  We need only to study what happens if $ x\in \mathcal A_1^c(a/2)$, $\epsilon \in \mathcal A_1(a)$ and $|\epsilon/x-1| \leq M \sqrt{\log z /z}$. We assume that $z$ is large enough so that $M \sqrt{\log z /z} \leq 1/2$ and hence
$$
x \leq \frac{\epsilon}{1 - M \sqrt{\log z /z}} \leq 2\epsilon.
$$
For  $\epsilon \in \mathcal A_1(a)$ and $|\epsilon/x-1| \leq M \sqrt{\log z}/\sqrt{z}$, we have
\begin{eqnarray*}
|x^j \tilde{h}^{j}(x)| &=& x^j |  \tilde{h}^{(j)}(\epsilon) + \sum_{l=1}^{r-j} \epsilon^{\ell}  \tilde{h}^{(j+\ell)}(\epsilon) \frac{(x/\epsilon-1)^{\ell} }{ \ell! }+ O((x/\epsilon-1)^{\beta-j}  L_{\tilde{h}^{j}}(\epsilon)  )|\\
&\leq & 2^j \left[ | \epsilon^j \tilde{h}^{(j)}(\epsilon)| + \sum_{l=1}^{r-j} |\epsilon^{\ell+j}  \tilde{h}^{(j+\ell)}(\epsilon)| \frac{M^\ell (\log z)^{\ell/2}}{ \ell!  z^{\ell/2} }\right]  \\
 && \quad +   O([z^{-1}\log z]^{(\beta-j)/2}) 2^j | L_{\tilde{h}^{j}}(\epsilon)|\\
&\leq & 2^j \tilde{h}(\epsilon) \delta z^{j/2} (\log z)^{-a} \left[1 +  \sum_{l=1}^{r-j} \frac{M^\ell (\log z)^{\ell/2} }{ \ell! }+ O([ \log z]^{(\beta-j)/2})\right]\\
&\leq & C_j \tilde{h}(\epsilon) \delta z^{j/2}[ (\log z)^{-a} + (\log z)^{-a + (\beta-j)/2}].
\end{eqnarray*}
We bound $\tilde{h}(x)$ from below using $\epsilon \in \mathcal A_1(a)$ and and $|\epsilon/x-1| \leq M \sqrt{\log z}/\sqrt{z}$:
\begin{equation*}
\begin{split}
\tilde{h}(x) &\geq \tilde{h}(\epsilon) - \sum_{j=1}^r \frac{(M\sqrt{\log z/z})^j }{ j! } x^j\tilde{h}^{(j)}(\epsilon) - 2 (M\sqrt{\log z/z})^\beta x^\beta L_{\tilde{h}}(\epsilon) \\
&\geq  \tilde{h}(\epsilon) \left( 1 - \sum_{j=1}^r \frac{(M\sqrt{\log z/z})^j }{ j! } \frac{ (2\epsilon)^j \tilde{h}^{(j)}(\epsilon)}{ \tilde{h}(\epsilon) } - 2 (M\sqrt{\log z/z})^\beta\frac{ (2\epsilon)^\beta L_{\tilde{h}}(\epsilon) }{ \tilde{h}(\epsilon )}\right) \\
&\geq  \tilde{h}(\epsilon) \left( 1 - \sum_{j=1}^r \frac{(2M)^j  }{ j! } \delta (\log z)^{-a+j/2} - 2 (2M)^\beta  \delta (\log z)^{-a+\beta/2}\right) \\
&\geq \tilde{h}(\epsilon) ( 1 - O((\log z)^{-a+ \beta/2}))
\end{split}
\end{equation*}
which implies that
\begin{eqnarray*}
\tilde{f}(x) &=& C_z^{\alpha} x^{\alpha-1}\tilde{h}(x) \geq C_z^{\alpha} (2\epsilon)^{\alpha-1}\tilde{h}(\epsilon) ( 1 - O((\log z)^{-a+ \beta/2}))\\
 &=& 2^{\alpha-1}\tilde{f}(\epsilon)( 1 - O((\log z)^{-a+ \beta/2}))
\end{eqnarray*}
and
$$ \frac{ | x^j \tilde{h}^{(j)}(x)| }{ \tilde{h}(x) } \leq \delta z^{j/2} (\log z)^{-a/2}[ (\log z)^{-a/2+ \beta/2} + 1].$$
 In particular, it implies that if $a \geq \beta$ then $x \in \mathcal A_1(a/2)$; so for $x \in \mathcal A_1^c(a/2)$, we must have $a<\beta$.
Therefore, using \eqref{FEzc2} and taking $a \in (\beta/2, \beta)$, we have
\begin{eqnarray*}
&& \int \int \1_{| \epsilon/x-1| \leq M \sqrt{\log z}/\sqrt{z}}\1_{\mathcal A_1^c(a/2)}(x) \1_{\mathcal A_1(a)} (\epsilon) g_{z,\epsilon}(x) \tilde{f}(\epsilon) d\epsilon dx \\
&\leq& 2^{\alpha-1}( 1 - O((\log z)^{-a+ \beta/2}))^{-1}(1+O(1/z))^{-1}
 \int \int  \1_{\mathcal A_1^c(a/2)}(x) \phi_{1/\sqrt{z}}(1 - u) \tilde{f}(x)  du dx\\
&\leq & 2^{\alpha-1}( 1 + O((\log z)^{-a+ \beta/2}))(1+O(1/z))  \tilde{F}(\mathcal A_1^c(a/2))\\
 &\lesssim&  z^{-\beta-e/4} ( 1 + O((\log z)^{-a+ \beta/2}))(1+O(1/z)),
\end{eqnarray*}
and \eqref{Hell:contdens} is proved.

\subsection{Adjustments for an unbounded density}\label{sec:ProofsAdjAlpha}

\begin{lemma}\label{lem:ApproxAlpha}
 For any $f\in \mathcal P_{\alpha}( \beta, L(\cdot), \gamma, C_0, C_1,e,\Delta)$, $x>0$, $A\in(0,\infty]$,
\begin{eqnarray*}
 \int_0^A g_{z,\epsilon}(x) f(\epsilon)d\epsilon =x^{\alpha-1}  \frac{z^{\alpha}\Gamma(z-\alpha)}{\Gamma(z)}\int_0^A  g_{z-\alpha+1,\epsilon}(x/ C_z)  h(\epsilon) d\epsilon,\\
  \int_0^A g_{z,\epsilon}(x) C_z f(C_z \epsilon)d\epsilon =x^{\alpha-1}  \frac{C_z^{\alpha-1}z^{\alpha}\Gamma(z-\alpha)}{\Gamma(z)}\int_0^{C_z A}  g_{z-\alpha+1,\epsilon}(  x)  h(\epsilon) d\epsilon,\\
\end{eqnarray*}
 and
 $$
 \frac{z^{\alpha}\Gamma(z-\alpha)}{\Gamma(z)}=1+O(1/z) \quad  \&  \,\, C_z^{\alpha-1}=1+O(1/z) \quad  \text{  as }  z\to \infty,
 $$
 where  $h(x) = x^{1-\alpha}f(x)$ and $C_z = \frac{z-\alpha+1}{z}$.
\end{lemma}

\begin{proof}[Proof of Lemma~\ref{lem:ApproxAlpha}]

Let $f\in \mathcal P_{\alpha}( \beta, L(\cdot), \gamma, C_0, C_1, e, \Delta)$ and denote $h(x) = x^{1-\alpha} f(x)$.

1. For  large enough $z$ and for any $A\in(0,\infty]$,  denoting $x_\alpha = C_z x$, we have
\begin{eqnarray*}
&&\int_0^A g_{z,\epsilon}(x) f(\epsilon)d\epsilon =  \frac{ 1 }{ \Gamma(z) }\int_0^A  x^{z-1}e^{ -zx/\epsilon} \left( \frac{ z}{ \epsilon }\right)^{ z } \epsilon^{\alpha-1} h(\epsilon) d\epsilon \\
&=&  x^{\alpha} \frac{ \Gamma(z -\alpha+1) }{ \Gamma(z) } \left(x_{\alpha}  \right)^{ z -\alpha+1}  \int_0^A  \frac{e^{ -( z -\alpha+1)x_{\alpha}/\epsilon}}{\Gamma(z -\alpha+1)} \left(  \frac{ z -\alpha+1}{ \epsilon }\right)^{ z -\alpha+1}   h(\epsilon) d\epsilon\\
&=&  x^{\alpha-1}  \frac{ z^{\alpha}\Gamma(z -\alpha+1) }{(z-\alpha+1) \Gamma(z) }  \int_0^A  g_{z-\alpha+1,\epsilon}(x_\alpha)  h(\epsilon) d\epsilon\\
&=&  x^{\alpha-1}  (1+O(1/z))\int_0^A  g_{z-\alpha+1,\epsilon}(x_\alpha)  h(\epsilon) d\epsilon
\end{eqnarray*}
since the Stirling formula implies
$$
\frac{ \Gamma(z -\alpha) }{\Gamma(z) }
 =  z^{-\alpha}(1+O(1/z)).
$$

2. For  large enough $z$ and for any $A\in(0,\infty]$,
\begin{eqnarray*}
&&\int_0^A g_{z,\epsilon}(x) C_z f(C_z \epsilon)d\epsilon =  \frac{ x^{z-1} }{ \Gamma(z) }\int_0^A  e^{ -zx/\epsilon} \left( \frac{ z}{ \epsilon }\right)^{ z } (C_z \epsilon)^{\alpha-1} h(C_z \epsilon) C_z d\epsilon\\
 &=&[v=C_z \epsilon]
=  \frac{ x^{z-1} C_z^z}{ \Gamma(z) }\int_0^{C_z A}  e^{ -zx C_z/v} \left( \frac{ z}{ v }\right)^{ z } v^{\alpha-1} h(v)  dv \\
&=&  x^{\alpha-1}\frac{  z^{\alpha-1} C_z^{\alpha-1} \Gamma(z+1-\alpha)}{ \Gamma(z) }  \frac{x^{z-\alpha}}{\Gamma(z+1-\alpha)} \\
&& \quad \times \int_0^{C_z A}  e^{ -  (z+1-\alpha)x /  v} \left( \frac{ z+1-\alpha}{ v }\right)^{ z +1-\alpha}   h(v)  dv \\
&=&  x^{\alpha-1}\frac{  z^{\alpha-1} C_z^{\alpha-1} \Gamma(z+1-\alpha)}{ \Gamma(z) }    \int_0^{C_z A}  g_{ z +1-\alpha,\epsilon}(x)   h(\epsilon)  d\epsilon.
\end{eqnarray*}
Since  $C_z= 1-\frac{ 1-\alpha }{z -\alpha+1}=1+O(z^{-1})$, we also have
 $C_z^{\alpha-1} = (1+(1-\alpha)/z)^{\alpha-1} = 1+O(1/z)$ for large $z$.

Therefore, the lemma is proved.
\end{proof}



\section{Approximation of densities by finite mixtures} \label{app:finite:approx}

\subsection{Construction of the discrete approximation}

The construction of a discrete finite mixture and the lower bound on the prior mass of Kullback-Leibler neighbourhoods of a smooth density $f$ are  similar to \citet{ghosal:vdv:mixture:07} and \citet{rousseau:09}. We first present the construction of the discrete distribution in Lemma \ref{lem:discrete1}, then we control the Hellinger distance between $f$ and the discrete approximation in Lemma \ref{lem:KL:PN}.
\begin{lemma} \label{lem:discrete1}
Let $e_z = z^{-a}$,  $E_z = z^b$ and $H$ be a probability distribution on $[e_z, E_z]$. Then for all $\kappa>0$,  there exists $N_0>0$ and a probability distribution $P$ with at most $\bar N = N_0 \sqrt{z} (\log z)^{3/2}$ supporting points such that : for all $x\in [\tau_0 e_z, \tau_1E_z]$ with $0<\tau_0<1 < \tau_1 <+\infty$
 \begin{equation}\label{discrete:1}
\left|  K_z*(H-P) (x) \right| \leq z^{-\kappa}, \quad \mbox{ when } z \mbox{ is large enough}
 \end{equation}
\end{lemma}

\begin{proof}[Proof of Lemma \ref{lem:discrete1}]

The proof of the lemma is based on the ideas of \citet{ghosal:vdv:01} combined with some ideas of  \citet{rousseau:09}. We use Gaussian approximation \eqref{gaussian:approx}. Set $u =\epsilon/x$ and consider $|u-1 | \leq M\sqrt{\log z/z}:= \delta_z$ with $M$ some arbitrarily large constant.
Then writing $h(u)  =   \log u +1/u -1$
\begin{equation*}
\begin{split}
\frac{ z^z e^{- z/u}}{ \Gamma(z)u^{z}} &= \frac{\sqrt{z} \exp\left( - z [ \log u +1/u -1] \right) }{ \sqrt{2\pi}} \left( 1 + R(z) \right)^{-1} \\
&= \frac{ \sqrt{z}}{ \sqrt{2\pi} ( 1 + R(z) )} \left( 1 +\sum_{j=1}^N \frac{ z^j h(u)^j }{ j! } + R_N(u) \right)
\end{split}
\end{equation*}
where $| R_N(u)|  \leq \frac{ z^{N+1} h(u)^{N+1} }{ (N+1)! }  $.
Note that $0\leq h(u) \leq ( u-1)^2 $ when $|u-1 | \leq \delta_z$ and $z$ is large enough. Hence for all, $u$ such that $|u-1 | \leq \delta_z$,
\begin{equation*}
\begin{split}
| R_N(u)| & \leq \frac{ ( M \log z)^{(N+1)} }{ (N+1)! }  \leq \frac{ e^{- ( N+1)  \log \left( \frac{N+1}{ M e \log z } \right)}}{\sqrt{N+1} } \leq e^{ - \tau ( N+1) } ,
\end{split}
\end{equation*}
as soon as $N+1 > \tilde{N}_0 \log z $ with $\tilde{N}_0$ large enough, for some $\tau>0$. Choose $r\geq 2$, then a Taylor expansion of $h(u) $ around $1$ leads to
\begin{equation*}
h(u) =  \frac{( u-1)^2}{2} +  \sum_{j=3}^{r} h_j ( u-1)^j + R_h(u), \quad |R_h(u) | \leq \frac{ C_h |u-1|^{r+1}}{ (r+1)! }
\end{equation*}
with $h_j=h^{(j)}(1)$.
This implies that
\begin{equation*}
g_{z,\epsilon}(x) =\frac{ Q_{N,z}(\epsilon, x) }{ x } +\Delta_N(x) ,
\end{equation*}
where $Q_{N,z}(\epsilon, x ) $ is a polynomial function of $\epsilon$   with degrees less than or equal to $r N$ and
$$\Delta_N(\epsilon, x) =  \frac{ R_N( \epsilon/x )}{x } + O\left(  \frac{ N\sqrt{z} [| zR_h(\epsilon/x)|  + |z^NR_h(\epsilon/x)|^N /N! ]}{x}\right).$$
For all $|\epsilon /x -1|\leq \delta_z$,
$$\Delta_N(\epsilon, x) \lesssim \frac{e^{ - \tau ( N+1) }}{x}   + \frac{z^{-r/2+1}(\log z)^{(r+1)/2} }{ x }. $$
 If $|\epsilon/x -1 | \in ( \delta_z , \delta)$ with $\delta >0$ arbitrarily small but fixed, $h(\epsilon/x)\geq (\epsilon/x-1)^2 /4$ and
 \begin{equation*}
 g_{z,\epsilon}(x) \leq  \frac{2\sqrt{z} \exp\left( - \frac{z (\epsilon-x)^2}{ 4 x^2 }  \right) }{ x\sqrt{2\pi}} .
 \end{equation*}
  Split $[e_z, E_z]$ into intervals in the form
 $ I_j = [e_z(1+ \delta_z/2)^j , e_z (1+ \delta_z/2)^{j+1}]$, with $j \leq J_z$ and
 $$J_z  = \lceil  \log ( E_z /e_z) / \log (1+ \delta_z/2)\rceil \lesssim \frac{(b+a)}{ \delta_z } \log z \lesssim \sqrt{z}( \log z)^{1/2} $$
 Following Lemma A1  of  \citet{ghosal:vdv:01}, since the functions $\epsilon \rightarrow \epsilon^{\ell}$, $\ell \leq rN$ are continuous over $I_j$, there exists a probability $P_{j,N}$ with  support included in $I_j$ with at most $rN+1$ points in the support such that for all $\ell\leq rN$
\begin{equation}\label{eq:MatchPH}
\int_{I_j} \epsilon^{\ell} d H_j(\epsilon) = \int_{I_j} \epsilon^{\ell} dP_{j,N}(\epsilon),
 \end{equation}
 where $H_j  = H\1_{I_j} / H[I_j]$. Construct
  $ P_N = \sum_j H(I_j) P_{j,N}$, then $P_N$ has support $[e_z, E_z]$ and for all $x$,
  \begin{equation*}
|  K_z*H(x)  - K_z*P_N(x) | \leq  \sum_{j=1}^{J_z} H(I_j) \sup_{\epsilon\in I_j} |\Delta_{N}(\epsilon, x)| .
  \end{equation*}
Let $\epsilon \in I_j$ and $x\in [e_z(1+ \delta_z/2)^{j-1}, e_z(1+ \delta_z/2)^{j+2}]$, then $x/\epsilon \leq (1+ \delta_z/2)^2 \lesssim 1+ 2 \delta_z$ when $z$ is large enough and
$$|\Delta_{N}(\epsilon, x)| \leq  \frac{e^{ - \tau ( N+1) }}{x}   + \frac{z^{-r/2+1}(\log z)^{(r+1)/2} }{ x } \leq z^{-\kappa-1}$$
as soon as $r/2 > a + \kappa +1$. If $|\epsilon/x -1|\in ( 2\delta_z, \delta) $
\begin{equation*}
\begin{split}
|  K_z*{H_j}(x)  - K_z*P_{j,N}(x) | &\leq   K_z*{H_j}(x)  + K_z*P_{j,N}(x) \\
&\leq  \frac{4\sqrt{z} \exp\left( - \frac{M^2 \log z }{ 4  }  \right) }{ x\sqrt{2\pi}} \lesssim z^{-(M^2-2)/4+a}  \leq z^{-\kappa}
\end{split}
\end{equation*}
as soon as $M^2 > 4(\kappa+a)+2$. Finally if $|\epsilon/x-1| > \delta$, using Lemma \ref{lem:tech1}
\begin{equation*}
|  K_z*{H_j}(x)  - K_z*P_{j,N}(x) | \leq   K_z*{H_j}(x)  + K_z*P_{j,N}(x) \lesssim e^{-cz }
\end{equation*}
for some $c>0$.  This implies that for all $x \in \R$,
  \begin{equation*}
|  K_z*H(x)  - K_z*P_N(x) | \leq  z^{-\kappa}.
  \end{equation*}
where $P_N$ has at most $N_0\sqrt{z}(\log z)^{3/2} $ supporting points in $[e_z,E_z]$, with $N_0$ depending on $\kappa, a, b$.
\end{proof}

The following Lemma allows us to control the Kullback-Leibler divergence between $f$ and $K_z*P$.
\begin{lemma} \label{lem:KL:PN}
Assume that $f\in \mathcal P_{\alpha}(\beta, L, \gamma, C_0, C_1, e , \Delta)$,  and that there exist $C_2>0$ and  $\rho_1>0$  such that
 $$
 \int_{x}^\infty y^2 f(y) dy  \leq C_2 (1+x)^{-\rho_1}.$$
 Let $e_z = z^{-a}$ and $E_z = z^b$ with $a >  (\alpha \wedge 1)^{-1} (\beta \vee (2 \beta - 1))$ and $b > (\beta \vee (2\beta -1))/(\rho_1+2) $. Then there exists
 $$P_N  = \sum_{i=1}^N p_i \delta_{(u_i)}, \quad u_i \in [e_z, E_z],\quad  N \leq N_0 \sqrt{z} (\log z)^{3/2} $$
 such that
 $$  D_H^2(K_z*P_N, K_z*\bar f_\beta) \lesssim z^{-\beta} , \quad D_H^2(K_z*P_N, f)\lesssim z^{-\beta}.$$
 Moreover,  there exists $A>0$ such that we can choose  $u_1\leq \cdots \leq u_N$, $u_i - u_{i-1} > z^{-A} $ and $p_i > 3 z^{-A}$ for all $i \leq N$ as long as $z$ is large enough.
\end{lemma}

Note that it appears from the proof of Lemma \ref{lem:KL:PN}, that $P_N$ can be chosen so that for all $1 \leq \ell \leq J_z$ where $J_z$ is such that
$e_z( 1 + M z^{-1/2}\sqrt{\log z})^{J_z+1} \geq  E_z $ and $e_z( 1 + M z^{-1/2}\sqrt{\log z})^{J_z} \leq  E_z $, $P_N(U_\ell)\geq  3z^{-A}$.

\begin{proof}[Proof of Lemma \ref{lem:KL:PN}]
Consider $\bar f_\beta$  as defined in Proposition~\ref{lem:approx:1}. First we approximate this function by a  function supported on $[e_z, E_z]$ so that both upper and lower approximation errors are bounded by $O(z^{-\beta})$. Recall that  $\tilde h(x)  = C_z (C_z x)^{1 - \alpha} f_\beta (C_z x) $.
 Since $\int_0^{e_z} f(\epsilon) d\epsilon \lesssim e_z^{\alpha}$ for small $e_z$, then, by definition of $f_\beta$, we have
\begin{equation*}
\begin{split}
\int_0^{e_z} \bar f_\beta(\epsilon  ) d\epsilon & \lesssim e_z^\alpha + \sum_{j=1}^{r} \frac{\int_0^{e_z}\epsilon^{j+\alpha-1} |\tilde{h}^{(j)}(\epsilon)|d\epsilon }{ z^{j/2} }\\
&\lesssim e_z^\alpha +  \sum_{j=1}^{r}  \sqrt{e_z^\alpha} z^{-j/2} \left(  \int_0^{+\infty } \frac{(\epsilon\wedge 1 )^{2j+\alpha-1}|\tilde{h}^{(j)}(\epsilon)|^2 }{ \tilde{h}(\epsilon)} d\epsilon\right)^{1/2}  \\
&\lesssim e_z^\alpha + \frac{ e_z^{\alpha/2}}{ \sqrt{z}} I(r>0).
\end{split}
\end{equation*}
We also have that
\begin{equation*}
\begin{split}
\int_{E_z}^{+\infty} \bar f_\beta(\epsilon) d\epsilon &\lesssim 1 - F(E_z) + \sum_{j=1}^r  \frac{\int_{E_z}^{+\infty}\epsilon^{j+\alpha-1} |h^{(j)}(\epsilon)|d\epsilon }{ z^{j/2} }\\
&\lesssim 1 -F(E_z) + \frac{ \sqrt{1 - F(E_z)} }{ \sqrt{z}}
\end{split}
\end{equation*}

Therefore, for  $e_z = z^{-a}$,  $\bar F_\beta([0, e_z])\leq C z^{-\beta}$ since  $a > \alpha^{-1} ( \beta\vee (2\beta - 1)) $.
For  $E_z = z^b$, inequality $\bar F_\beta([E_z,\infty))\leq C z^{-\beta}$ is satisfied with $b > ( \beta\vee (2\beta - 1))//(\rho_1+2) $.
We thus have that $\bar F_\beta[e_z,E_z] \geq 1 - O(z^{-\beta})$.
Define
$$ f_\beta = \frac{\bar f_\beta \1_{[e_z,E_z]}}{\bar F_\beta[e_z,E_z]}$$
then
\begin{eqnarray*}
K_z \bar f_\beta(x) &=& \bar F_\beta([e_z, E_z])K_z f_\beta(x) + \int_0^{e_z} \bar{  f}_\beta( \epsilon) g_{z,\epsilon}(x)d\epsilon \\
&  & + \int_{E_z}^{+\infty} \bar f_\beta( \epsilon) g_{z,\epsilon}(x)d\epsilon \geq \bar F_\beta([e_z, E_z])K_z f_\beta(x).
\end{eqnarray*}
 It implies that
\begin{eqnarray*}
D_H^2(K_z {\bar f}_\beta, K_z  f_\beta) \leq 2 - 2 \int_0^\infty K_z f_\beta(x)dx \sqrt{\bar F_\beta([e_z, E_z])} \\
\leq 2-2\sqrt{1-O(z^{-\beta})} \lesssim z^{-\beta},
\end{eqnarray*}
so that
\begin{equation}\label{f:tildef2}
D_H^2(f, K_z f_\beta ) \leq [D_H(f, K_z  \bar f_\beta) + D_H(K_z \bar f_\beta, K_z  f_\beta)]^2 \lesssim z^{-\beta}.
\end{equation}
For an arbitrary $\kappa>0$, which we will choose later, consider the discrete distribution $P_N$ constructed in Lemma \ref{lem:discrete1}, which we write as
 $\tilde P_N  = \sum_{j=1}^{J_z} \sum_{i=1}^{N_j} p_{j,i} \delta_{\epsilon_{j,i}}$, $\epsilon_{j,i} \in [e_z(1+\delta_z/2)^j,e_z(1+\delta_z/2)^{j+1}]$, with $N= \sum_{j=1}^{J_z}N_j \leq N_0 \sqrt{z}(\log z)^{3/2}$ and $\delta_z = M z^{-1/2}\sqrt{\log z}$, where  $N_0 = N_0(\kappa)$ such that
 \begin{equation}\label{approx2}
|K_z*\tilde P_N (x) - K_z  f_\beta (x) | \lesssim z^{-\kappa} \quad \forall x \in [e_z/2, 2 E_z].
 \end{equation}
 Note that for $e_z= z^{-a}$ and $E_z = z^b$,  $J_z \leq (b+a)M^{-1} \sqrt{z \log z}$  and $N_j \lesssim \log z$.  This implies that 
 \begin{equation*}
 \begin{split}
  D_H^2(K_z*\tilde P_N, K_z f_\beta) & \leq \frac{ 1}{2 }\|K_z*\tilde P_N - K_z f_\beta\|_1\\
  &  \lesssim z^{-\kappa} (2E_z-e_z/2) + \int_{[e_z/2,2E_z]^c} (K_z*\tilde P_N+K_z f_\beta)(x)dx.
\end{split}
\end{equation*}
For any distribution $P$ with support $[e_z, E_z]$, by Lemma~\ref{lem:tech1} with  $u=\epsilon/x >1+\delta$, $\delta=1$ and $c_{1}= c(\delta)$:
\begin{eqnarray*}
\int_{0}^{ e_z/2} (K_z*P)(x)dx &=& \int_{0}^{ e_z/2} dx\int_{[e_z, E_z]} g_{z,\epsilon}(x) dP(\epsilon)\\
&\leq& \int_{0}^{e_z/2} x^{c(\delta)z-1} dx\int_{[e_z, E_z]} \epsilon^{-c(\delta)z} dP(\epsilon) \\
&\lesssim&    z^{-1} (e_z/2)^{c(\delta)z} e_z^{-c(\delta)z} \int_{[e_z, E_z]}  dP(\epsilon) \lesssim    2^{-c_1 z} z^{-1}.
\end{eqnarray*}
Similarly, applying Lemma~\ref{lem:tech1} with $u=\epsilon/x < 1-\delta$, $\delta=1/2$  and $c_{0.5}= c(\delta)$,
\begin{eqnarray*}
\int_{2E_z}^{\infty} (K_z*P)(x)dx &=& \int_{2E_z}^{\infty} dx\int_{[e_z, E_z]} g_{z,\epsilon}(x) dP(\epsilon)\\
&\leq& \int_{2E_z}^{\infty} x^{-1} e^{-c(\delta)z x/E_z} dx\int_{[e_z, E_z]}  dP(\epsilon) \\
&\lesssim&    E_z^{-1} \frac{e^{-2c(\delta)z }}{z /E_z}  \int_{[e_z, E_z]}  dP(\epsilon) \lesssim   z^{-1} 2^{-2c_{0.5} z}.
\end{eqnarray*}

Hence choosing $\kappa \geq 2 \beta + b$ implies that
 $$  D_H^2(K_z*\tilde P_N, K_z {\bar f}_\beta) \leq z^{-\beta} , \quad D_H^2(K_z*\tilde P_N, f)\lesssim z^{-\beta}.$$

Let $A>0$ and construct the grid $(u_\ell)_\ell$:
 $$ u_\ell = e_z ( 1 + z^{-A})^\ell, \ell = 0, \cdots L , \quad L = \left\lceil \frac{ \log E_z - \log e_z }{ \log ( 1 + z^{-A}) } \right\rceil \lesssim z^A \log z$$
 Let $ P_N=\sum_{j=1}^{J_z} \sum_{i=1}^{N_j} p_{j,i} \delta_{u_{j,i}}$ be the probability on  $\R^+ $ with supporting points $u_{j,i}$ where $u_{j,i}$ is the closest point to $\epsilon_{j,i}$ on the grid $(u_\ell, \ell\leq L)$.
 If there are multiple $u_{i,j}$ then we collapse the probabilities and without loss of generality we can assume that the $u_{i,j}$ are all distinct. Define
\begin{equation}\label{eq:DefU}
U(u_{\ell})  = [(u_{\ell}+u_{\ell-1})/2, (u_{\ell}+u_{\ell+1})/2],
\end{equation}
covering the interval $[e_z, E_z]$, with a suitable adjustment on the boundaries, and hence the corresponding sets $U_{j,i}=U(u_{j,i})$.
By construction $|u_{j,i}/\epsilon_{j,i}-1|\leq z^{-A}$ and  we have
first that
 for $x \in [e_z/2, 2E_z]$,
\begin{eqnarray*}
| K_z*P_N(x) - K_z*\tilde P_N(x)| &\leq&  \sum_{j=1}^{J_z} \sum_{i=1}^{N_j} p_{j,i}   |g_{z,u_{j,i}}(x) - g_{z,\epsilon_{j,i}}(x)| \\
&\leq& \sum_{j=1}^{J_z} \sum_{i=1}^{N_j} p_{j,i} g_{z,\epsilon_{j,i}}(x) \exp(z^{-A+1}(1+x/\epsilon_{j,i}))  \\
&\leq& \sum_{j=1}^{J_z} \sum_{i=1}^{N_j} p_{j,i} g_{z,\epsilon_{j,i}}(x) \exp(z^{-A+1}(1+2E_z/e_z))\\
&\leq& K_z*\tilde P_N(x) [1+ C z^{-A+1+b+a}]
\end{eqnarray*}
for large enough $z$, which implies
 \begin{equation}\label{upboundKP}
K_z* P_N (x) \leq K_z* \tilde P_N(x)( 2 + C z^{a+b+1-A}) , \quad \forall x \in [e_z/2, 2E_z].
\end{equation}
Finally,
\begin{eqnarray*}
D_H^2( K_z*P_N, K_z*\tilde P_N) &\leq& \frac 1 2 || K_z*P_N- K_z*\tilde P_N||_1 \\
& \leq & \frac 1 2 \sum_{j=1}^{J_z} \sum_{i=1}^{N_j} p_{j,i} \int |g_{z,u_{j,i}}(x) - g_{z,\epsilon_{j,i}}(x)| dx\\
&\leq&  \sqrt{1/2} \sum_{j=1}^{J_z} \sum_{i=1}^{N_j} p_{j,i}  z^{-A+1/2}  = z^{-A+1/2}/\sqrt{2},
\end{eqnarray*}
where the last inequality comes from Lemma \ref{L1:geps}. By choosing $A > 1/2 + \beta$, Lemma \ref{lem:KL:PN} is proved by re-indexing
$p_{i,j}$ as $p_l$  and $u_{i,j} $ as $u_l$, $l\leq N$.

\end{proof}

\subsection{ Kullback-Leibler neighbourhoods}  \label{sec:KLNeighbourhoods}
In the following Lemma we describe   Kullback-Leibler neighbourhoods of $f$ of size $\epsilon_n^2$.

\begin{lemma} \label{lem:KL:discrete2}
Assume that $f\in \mathcal P_{\alpha}(\beta, L, \gamma, C_0, C_1, e , \Delta)$,  and that there exists $C-2$ and $\rho_1>0$  such that
 $$
 \int_{x}^\infty y^2 f(y) dy  \leq C_2 (1+x)^{-\rho_1}.$$
Define $P_N = \sum_{i=1}^N p_i \delta_{u_i}$  and $A>0$ as in Lemma \ref{lem:KL:PN} and set
$$
\mathcal P_z = \{ P: \, P(U_{i})/p_i \in( 1 -  2z^{-2A}, 1 - z^{-2A})\, \forall i=1,\ldots,N \}.
$$
Then, if $A$ is large  enough, for all $z$ large enough and  all $P \in \mathcal P_z$,
 $$\KL( f , K_z*P) \leq z^{-\beta} (\log z) ; V( f, K_z *P)\leq z^{-\beta} (\log z)^2.$$
\end{lemma}
\begin{proof}[Proof of Lemma \ref{lem:KL:discrete2}]
Let $P_N$ be defined as in Lemma \ref{lem:KL:PN}. 
Using Lemma B2 of \citet{ghosal:shen:tokdar} with $\lambda=z^{-A_1}$ and $A_1 >0$ to be defined later, we have that if $P \in \mathcal P_z$,
\begin{equation*}
\begin{split}
\KL( f, K_z*P) &\lesssim D_H^2(f, K_z*P) ( 1 + A_1 \log z)\\
& \quad  + \int_{f> z^{A_1}K_z*P} f(x) \log  \left( \frac{f(x) }{K_z*P(x)}\right)dx \\
&\lesssim D_H^2(f, K_z*P) ( 1 + A_1 \log z) \\
& \quad -\int_{[e_z,E_z]\cap \{f> z^{A_1}K_z*P\} } \log(K_z*P(x)) f(x) dx  \\
& \quad + \int_{\{f> z^{A_1}K_z*P\} \cap [e_z, E_z]}f(x) ( \log f(x) )_+ dx \\
 & \quad + \int_{[e_z, E_z]^c\cap \{f> z^{A_1}K_z*P\}} f(x)  ( (\log f(x))_+ -\log(K_z*P(x))) dx,
 \end{split}
 \end{equation*}
 and similarly for $V( f, K_z*P)$. The above computations imply that for all $P \in \mathcal P_z$,   if $A\geq \beta$
$$D_H^2(f, K_z*P) \lesssim z^{-\beta}.$$

First, we show that for any $\kappa>0$, $\exists A, A_1$ such that if $ f(x)> z^{A_1}K_z*P(x)$ for  $x\in \mathcal A_1(0) \cap [e_z/2, 2E_z]$, then $f(x)\leq z^{-\kappa}$, where $\mathcal A_1(0)$ is defined in \eqref{def:A1}.
Using Lemma \ref{L1:geps},
 \begin{gather*}
| K_z*P(x) - K_z* P_N(x)| = |\int g_{z,u }(x) d P(u) - \sum_{i=1}^N p_{i} g_{z,u_{i} }(x)| \\
\leq \sum_{i=1}^N |\int_{U_{i}} g_{z,u_{i} }(x) d P(u) -  p_{i} g_{z,u_{i} }(x)| +\sum_{i=1}^N \int_{U_{i}} |g_{z,u }(x)-g_{z,u_{i} }(x)| d P(u)  \\
\leq \sum_{i=1}^N  g_{z,u_{i} }(x)  | P(U_{i}) -  p_{i}| +\sum_{i=1}^N  g_{z,u_{i} }(x) |e^{z^{1-2A}(x/u_{i}+1)}-1|  P(U_{i})  \\
\leq \sum_{i=1}^N  g_{z,u_{i} }(x) p_i  z^{-2A} +\sum_{i=1}^N  p_i g_{z,u_{i} }(x) |e^{z^{1-2A}(x/u_{i}+1)}-1|  (1 - z^{-2A})  \\
\leq  K_z* P_N(x) z^{-2A}   +\sum_{i=1}^N p_i g_{z,u_{i} }(x) |e^{z^{1-2A}(x/u_{i}+1)}-1|.
 \end{gather*}
 Then,
 \begin{equation}\label{eq:eq1}
 \begin{split}
 K_z*P(x) &\geq  K_z* P_N \left(1 -  z^{-2A}   \right)   -  \sum_{i=1}^N p_{i}g_{z,u_{i} }(x) |e^{  z^{-2A+1} (1+x  /u_{i})} - 1|.
 \end{split}
 \end{equation}


By construction $f_\beta \geq c_\beta \tilde{f}/2 \geq \tilde{f}/2 ( 1 + o(1) )$, also on $\mathcal A_1(0)$,
\begin{equation*}
K_z h (x ) \geq h(x) - \sum_{j=1}^r\frac{  | \mu_j h^{(j)}(x) x^j |}{ z^{j/2} } - z^{-\beta/2} |R_z(x) | \geq h(x)/2
\end{equation*}
which implies that
\begin{equation*}
K_z \tilde{f} (x ) = x^{\alpha-1}(1+O(1/z)) K_{z+1-\alpha} h (x)  \geq  (1+O(1/z)) f( x)/2. 
\end{equation*}
Using Lemma \ref{lem:discrete1}, on  $\mathcal A_1(0) \cap [e_z/2, 2E_z]$, for arbitrarily chosen $\kappa >0$,
$$f(x) \leq 2 K_z\tilde{f}(x) \leq 4 K_z f_\beta(x) \lesssim K_z*P_N + z^{-\kappa}. $$

Moreover, for $x\in  [e_z/2, 2E_z]$, using \eqref{eq:eq1},
 \begin{eqnarray*}
 \sum_{i=1}^N p_{i}g_{z,u_{i} }(x) |e^{  z^{-2A+1} (1+x  /u_{i})} - 1|\leq  \sum_{i=1}^N p_{i}g_{z,u_{i} }(x)  z^{-2A+1} (1+4 z^{a+b})
 \end{eqnarray*}
 which gives
\begin{equation}\label{eq:eq2}
  K_z*P(x) \geq  K_z* P_N \left(1 -  z^{-2A}    - 5 z^{-2A+a+b+1} \right)
\end{equation}
and hence  $K_z*P \geq C K_z*P_N$ as soon as $A > (a+b+1)/2$.
Hence on $\mathcal A_1(0) \cap [e_z/2, 2E_z]$, $f(x) > z^{A_1} K_z* P_N(x) $ with $A_1>0$ only if
$
f(x)\lesssim z^{-A_1}f(x)+ z^{-\kappa},
$
i.e. if   $f(x) \lesssim z^{-\kappa}$.

 Now consider $x$ such that  $f(x) \leq z^{A_1} K_z* P_N(x)$ and  $f(x) \geq 2 z^{A_1} K_z*P (x)$, i.e. such that $K_z* P_N(x) \geq 2  K_z*P (x)$. Then, using \eqref{eq:eq1},
\begin{equation*}
  2(1-2 z^{-A})^{-1} \sum_{j} p_{i,j}g_{z,u_{j} }(x) |e^{  z^{-A+1} (x /u_{j}+1)} - 1| \geq   K_z* P_N(x).
\end{equation*}
If $x/u_{j}\leq 2$ then $|e^{  z^{-A+1} (x /u_{j}+1)}-1| \lesssim z^{-A+1}$ while if $x> 2 u_{j}$, \,
$$g_{z,u_{j} }(x) e^{  z^{-A+1} (x /u_{j}+1)}  \leq e^{- c z (x/u_{j}+1)}/x$$ for some $c>0$.
Therefore,
\begin{eqnarray*}
 \sum_{j} p_{j}g_{z,u_{j} }(x) |e^{  z^{-A+1} (x /u_{j}+1)} - 1| &\leq& K_z* P_N z^{-A+1} +   \sum_{j} p_{j}\frac{ e^{- c z (x/u_{j}+1)}}{x} \1_{x > 2 u_{j}}\\
 &\leq& K_z* P_N z^{-A+1} +  e_z^{-1}  e^{- 3c z  }.
\end{eqnarray*}
Using \eqref{eq:eq1}, we have
\begin{eqnarray*}
K_z* P_N &\leq& \sum_{i=1}^N p_{i}g_{z,u_{i} }(x) |e^{  z^{-2A+1} (1+x  /u_{i})} - 1|  (0.5 -  z^{-2A})^{-1}\\
&\leq&  4 K_z* P_N z^{-A+1} +  4 e_z^{-1}  e^{- 3c z  }
\end{eqnarray*}
which in turn implies that
 $$   z^{a}  e^{- 3c z  }   \gtrsim  K_z* P_N \geq z^{-A_1} f(x) $$
 so that  $f(x) \lesssim  z^{a +A_1} e^{- 3c z  }$. In all cases, for all $H \geq \kappa$, by choosing $A$ and $A_1$ such that $A+1-a -A_1 > H  $ and so that \eqref{approx2} holds,  we obtain that on  $\mathcal A_1(0) \cap [e_z/2, 2E_z]$ if $2z^{A_1}K_z*P(x) \leq f(x)$ then $f(x) \leq z^{-H}$. For $x$ such that  $\mathcal A_1(0) \cap [e_z/2, 2E_z]$  and $z^{A_1}K_z*P(x) \leq f(x) \leq 2z^{A_1}K_z*P(x)$, $f(x) \leq z^{-\kappa}$.


Now we bound from below $K_z*P(x)$.
\begin{itemize}
\item  Take $x \in [e_z, E_z]$, and let $\ell$ be such that $x \in [e_z(1 + \delta_z)^\ell, e_z(1+\delta_z)^{\ell+1}] $
 with $e_z(1+\delta_z)^{\ell+1} \leq E_z$ and    $\delta_z=M\sqrt{ z^{-1 } \log z }$,  then
 \begin{equation*}
 \begin{split}
 K_z*P(x) &\geq P([e_z(1+\delta_z)^\ell, e_z(1+\delta_z)^{\ell+1}])  \frac{ \sqrt{z} e^{- z \delta_z^2/2} }{\sqrt{2\pi} }( 1 + o(1) )\\
 &\gtrsim z^{-A+1/2-M^2/2},
 \end{split}
 \end{equation*}
 since $P(U_j) \geq p_j ( 1 - z^{-A})$ for all $j$ and  $  P([e_z(1+\delta_z)^\ell, e_z(1+\delta_z)^{\ell+1}])\geq 3z^{-A}$.

\item If $x < e_z$,
 \begin{eqnarray*}
  K_z*P(x) &\geq& \frac{ z^{-A-2}e^{-zx/e_z(1+ \delta_z)} (zx/e_z(1+ \delta_z))^{z-1} }{ e_z\Gamma(z)} \\
  &\geq& \exp \left( 2 z\log (x/e_z) - c \log z\right),
 \end{eqnarray*}
 when $z$ is large enough, for some $c>0$.
\item  If $x>E_z$,
 \begin{eqnarray*}
  K_z*P(x) &\gtrsim& \frac{ e^{-zx/e_z} (zx/e_z)^{z-1} }{x\Gamma(z)} \geq \exp \left( - z\frac{x}{e_z} + (z-1)[\log (x)- \log (e_z)]\right)\\
   &\gtrsim& e^{-2zx/e_z}.
 \end{eqnarray*}
\end{itemize}

Then, using Lemma B2 of \citet{ghosal:shen:tokdar} with $\lambda=z^{-A_1}$, we have, using $\log f(x) \lesssim \log x$,
\begin{equation*}
\begin{split}
&\KL( f, K_z*P)  \lesssim D_H^2(f, K_z*P) ( 1 + A_1 \log z) \\
& \quad + \int_{f> z^{A_1}K_z*P} f(x) \log  ( f(x) /K_z*P(x))dx \\
&\lesssim D_H^2(f, K_z*P) ( 1 + A_1 \log z) + \log z \int_{[e_z,E_z]\cap \{f> z^{A_1}K_z*P\} } f(x) dx  \\
& \quad + \int_{\{f> z^{A_1}K_z*P\} \cap [e_z, E_z]}f(x) ( \log f( x))_+  dx\\
 & + z\int_{0}^{e_z} f(x)  ( |\log x| + \log z) dx + z^{a+1}\int_{E_z}^\infty f(x) x dx \\
 &\lesssim z^{-\beta} \log z  + z\log z F(0,e_z) +  z^{-b(2+\rho_1)/2} +z^{a+1-b( 1+\rho_1)} \\
& + \log z\int_{\mathcal A_1(0)\cap  [e_z,E_z]\cap \{f> z^{A_1}K_z*P\}}  f(x) dx\\
& \quad  +\log z\int_{\mathcal A_1(0)^c\cap  [e_z,E_z]\cap \{f> z^{A_1}K_z*P\}}  f(x) dx\\
 &\lesssim z^{-\beta} \log z  + z^{1-a\alpha}\log z   +  z^{-b(2+\rho_1)/2} +z^{a+1-b( 1+\rho_1)}   + z^{b-\kappa} \log z
+  z^{-\beta-e/4} \log z
\end{split}
\end{equation*}
using $F(\mathcal A_1(0)^c)\lesssim z^{-\beta-e/4}$ and
\begin{eqnarray*}
 \int_{E_z}^\infty f(x) (\log f(x))_+dx  &\leq& \sqrt{F[E_z,+\infty )}\sqrt{\int_{1}^\infty f(x) \log^2(x)dx } \lesssim E_z^{-1-\rho_1/2},    \\
 \int_{E_z}^\infty x^k f(x)  dx  &\leq&  E_{z}^{k-2}\int_{E_z}^\infty x^2 f(x) dx \lesssim   E_z^{k-2-\rho_1}
\end{eqnarray*}
  for $E_z>1$ and $k \in [0,2)$.

Choosing $a$, $b$ and $\kappa $ such that
$$
a  \geq (\beta+1)/\alpha, \quad b\geq 2\beta/(2+\rho_1),\quad  b( 1+\rho_1)-a-1\geq \beta \quad \kappa \geq b+ \beta,
$$
we  have that
$$ \int f  \log\left( \frac{ f }{ K_z*P} \right)  \1_{f > z^{A_1}  K_z*P }  \lesssim z^{-\beta}(\log z).$$

Similarly,
$$ \int f \left[\log\left( \frac{ f }{ K_z*P} \right)\right]^2\1_{f > z^{A_1}  K_z*P}  \lesssim z^{-\beta}(\log z)^2 ,$$
under the same constraints.

\end{proof}


\section{Some technical lemmas}
\begin{lemma}\label{L1:geps}
For all $\delta >0$, there exists $C>0$ such that for all $\epsilon_1, \epsilon_2$ satisfying $|\epsilon_1/\epsilon_2 -1 | < \delta$
\begin{equation*}
\begin{split}
\|g_{z,\epsilon_1}- g_{z, \epsilon_2}\|_1 \leq \sqrt{2\KL(g_{z,\epsilon_1}, g_{z,\epsilon_2})} \leq \sqrt{2z}\delta, \quad g_{z,\epsilon_2}(x)\leq g_{z,\epsilon_1}(x)e^{z\delta( 1 + x/\epsilon_1)}.
\end{split}
\end{equation*}

\end{lemma}

\begin{proof}[Proof of Lemma \ref{L1:geps}]
Inequality
 $$
 ||g_{z,\epsilon_1}-g_{z,\epsilon_2}||_1 \leq \sqrt{2\KL(g_{z,\epsilon_1}, g_{z,\epsilon_2})}
 $$
 holds due the inequality  for the total variation distance to be upper bounded by $\sqrt{2}$ times the square root of the Kullback-Leibler distance between the corresponding probability distributions.

The Kullback-Leibler distance between two densities $g_{z, \epsilon_1}$ and $g_{z,\epsilon_2}$ is
\begin{eqnarray*}
\KL(g_{z, \epsilon_1}, g_{z,\epsilon_2}) &=& \int_0^{\infty} g_{z,\epsilon_1}(x) \log\left( \frac{g_{z, \epsilon_1}(x)}{g_{z,\epsilon_2}(x)}\right) dx\\
&=& \int_0^{\infty} g_{z, \epsilon_1}(x)  \left(zx [\epsilon_2^{-1} - \epsilon_1^{-1}] + z \log( \epsilon_2/\epsilon_1) \right) dx\\
&=&  z\left( \epsilon_1/\epsilon_2 - 1 -  \log(\epsilon_1/\epsilon_2) \right) \\
&\leq& z \delta^2
\end{eqnarray*}
due to condition $|  \epsilon_1/\epsilon_2-1|<\delta$ and inequality $x-1-\log x \leq |x-1| r/(1+r)\leq r^2$ if $|x-1|\leq r$.

Moreover, for any $x>0$,
\begin{equation*}
\begin{split}
\left| \log g_{z,\epsilon_2} (x) - \log g_{z,\epsilon_1}(x) \right|&=
 \left|zx [\epsilon_1^{-1} - \epsilon_2^{-1}] + z \log( \epsilon_1/\epsilon_2) \right| \\
 & \leq  \left|\frac{ zx}{ \epsilon_1}\delta  + z \log(  1+ \delta) \right|
 \end{split}
\end{equation*}
which implies that
\begin{equation*}
\begin{split}
  g_{z,\epsilon_2} (x)/ g_{z,\epsilon_1}(x)  &\leq  \exp \left\{\frac{ zx}{ \epsilon_1}\delta  + z \delta  \right\}
 \end{split}
\end{equation*}
which completes the proof.
\end{proof}


 \subsection{Properties of gamma densities}\label{sec:tech}

 In this section we present some technical computations which are used throughout the paper.
We first present some identities on mixtures of Gamma densities, together with tail inequalities
 \begin{lemma}\label{lem:tech1}
 Let $z>0$ and $x>0$, then
  \begin{equation} \label{I0}
 I_0(z,x)  :=  \int_0^\infty g_{z,\epsilon}(x)d\epsilon  = 1 + \frac{ 1 }{ z-1}
 \end{equation}
and for all $k \geq 0$
 \begin{equation}\label{Ik}
 I_k(z,x)  :=  \int_0^\infty (\epsilon - x)^k g_{z,\epsilon}(x)d\epsilon = \frac{ x^{k} z^z }{ \Gamma(z) } \int_0^\infty \frac{( u -1)^ke^{- z/u} }{ u^{z} }du
 \end{equation}
Moreover for all $\delta \in (0,1)$ there exists $c(\delta)>0$ such that for all $z $ large enough and  $u < 1-\delta $,
\begin{equation}\label{lb:gz}
\frac{ z^z e^{- z/u}}{ \Gamma(z)u^{z}} \leq  e^{ - c(\delta) z /u }
\end{equation}
and for all  $u >1+ \delta$
\begin{equation}\label{ub:gz}
\frac{ z^z e^{- z/u}}{ \Gamma(z)u^{z}} \leq  u^{-c(\delta)z }.
\end{equation}
 \end{lemma}

\begin{proof}[Proof of Lemma \ref{lem:tech1}]
We have
  \begin{equation*}
 \begin{split}
 I_0(z,x) & :=  \int_0^\infty g_{z,\epsilon}(x)d\epsilon  = \frac{ x^{z-1} z^z }{ \Gamma(z) } \int_0^\infty e^{ -zx/\epsilon}\frac{ 1}{ \epsilon^{(z-1)+1 }}d\epsilon \\
 &= \frac{ x^{z-1} z^z }{ \Gamma(z) } (zx)^{- ( z-1)}\Gamma(z-1)  = \frac{ z }{ z-1 }  = 1 + \frac{ 1 }{ z-1},
 \end{split}
 \end{equation*}
 which proves \eqref{I0}. For all $k \geq 0$,
 \begin{equation*}
 \begin{split}
 I_k(z,x) & :=  \int_0^\infty (\epsilon - x)^k g_{z,\epsilon}(x)d\epsilon
           = \frac{ x^{k-1} z^z }{ \Gamma(z) } \int_0^\infty \left(\frac{\epsilon}{x} -1 \right)^ke^{ -zx/\epsilon}\frac{ x^{(z-1)+1} }{ \epsilon^{(z-1)+1 }}d\epsilon \\
 &= \frac{ x^{k} z^z }{ \Gamma(z) } \int_0^\infty \frac{( u -1)^ke^{- z/u} }{ u^{z} }du
 \end{split}
 \end{equation*}
and \eqref{Ik} is verified.
Now, note that when $z$ is large
\begin{equation} \label{gaussian:approx}
\begin{split}
\frac{ z^z e^{- z/u}}{ \Gamma(z)u^{z}} &= \frac{\sqrt{z} \exp\left( - z [ \log u +1/u -1] \right) }{ \sqrt{2\pi}} \left( 1 + R(z) \right)^{-1} \\
&= \frac{\sqrt{z} \exp\left( - \frac{z}{2} (1- u)^2( 1 + o(1))  \right) }{ \sqrt{2\pi}}\left( 1 + R(z) \right)^{-1}
\end{split}
\end{equation}
where $R(z) =O(1/z)$ is the remainder term of the Stirling formula. When $u< 1-\delta$ the first inequality leads to \eqref{lb:gz}  while when $u > 1+\delta$ it leads to \eqref{ub:gz}.
\end{proof}

From Lemma \ref{lem:tech1}, we can deduce the following approximations:
\begin{lemma}   \label{lem:tech2}
 For all $k\geq 0$ and $x>0$,
\begin{equation}\label{Ik}
I_k(z,x) =\frac{  x^k }{ z^{k/2} }  \left( 1 + R(z) \right)^{-1}  \left(  \mu_k  + O(z^{-H })\right) := \frac{  x^k }{ z^{k/2} } \mu_{k}(z) , \quad \forall H >0,
\end{equation}
where $\mu_k = \int_{\R} x^k \varphi(x) dx$ with $\varphi$ the density of a standard  Gaussian random variable. We also have
\begin{equation}\label{eq:Kzf}
K_z f(x)  =  \sum_{j=0}^r \frac{f^{(j)}(x)x^j }{ j ! z^{j/2}}  \mu_{j}(z)  + z^{-\beta/2} R_z(x)
\end{equation}
where
$$
| R_z(x)| \leq C_{\beta,z} L(x) x^{\beta}   \left[ 1 + \frac{  x^{\gamma} }{ z^{ \gamma/2 }} \right].
$$
For all  $g(x)  \leq C_1  + C_2 x^a $ for some $a>0$, then
\begin{equation}\label{ineq:KR}
K_zg(x)    \leq 2 C_1 + 2C_2 x^a,
\end{equation}
for $z$ large enough and $a$ fixed.
\end{lemma}

\begin{proof}[Proof of Lemma \ref{lem:tech2}]
Lemma \ref{lem:tech1} implies that
\begin{equation}\label{eq:Kzf}
K_z f(x)  =  \sum_{j=0}^r \frac{f^{(j)}(x)}{ j !}  I_j(z,x)  + R_z(x) =  \sum_{j=0}^r \frac{f^{(j)}(x)x^j }{ j ! z^{j/2}}  \mu_{j}(z)  + z^{-\beta/2} R_z(x)
\end{equation}
where
$$
| R_z(x)|\leq C_{\beta} L(x) z^{\beta/2}[ I_{\beta}(z,x) + I_{\beta+\gamma}(z,x)] \leq C_{\beta,z} L(x) x^{\beta}   \left[ 1 + \frac{  x^{\gamma} }{ z^{ \gamma/2 }} \right].
$$

Also if $g(x) $ be a function bounded by $C_1  + C_2 x^a $ for some $a>0$, then
\begin{equation*}
\begin{split}
K_z g(x) &\leq C_1 I_0(z) + C_2  \frac{ x^{z-1} z^z }{ \Gamma(z) } \int_0^\infty e^{ -zx/\epsilon}\frac{ 1}{ \epsilon^{(z-a-1)+1 }}d\epsilon \\
&\leq 2 C_1 + C_2 \frac{ x^{a} \Gamma(z-a-1)z^z }{ \Gamma(z) z^{z-a-1}}   \leq 2 C_1 + 2C_2 x^a,
\end{split}
\end{equation*}
for $z$ large enough and $a$ fixed.
\end{proof}

 \subsection{Examples of functions in $\mathcal P_\alpha ( \beta, L(\cdot), \gamma, C_0, C_1, e, \Delta )$}\label{sec:VerifCond}

In this section we verify conditions in Remark~\ref{rk:functionalclass}. 

In Remark~\ref{rk:functionalclass} we state that moment condition \eqref{moment:func} is satisfied for Weibull distribution $f(x) = C x^{\alpha-1} e^{-c x^{b}}$ with $\alpha, b,c > 0$; for folded Student t distribution $f(x) = c_\nu (1+x^2)^{-(\nu+1)/2}$, $x>0$; for the Frechet-type distributions $f(x)= b x^{-b-1} e^{-x^{-b}}$, $b>0$.


 \subsubsection{Weibull distribution}

Consider Weibull distribution with density $f(x) = C_{a, b} x^{a-1} e^{- x^{b}}$ with $a, b>0$. Assume first that $a=\alpha \in (0,1]$ and  $b\geq 1$, then  $\ell(x) = \log h(x) = - x^{b} +\log C_{\alpha, b}$ which is infinitely differentiable. Take some integer $r\geq 0$ such that $b-r \in (0,1]$ then
$\ell^{(r)}(x)= -(b)_r x^{b-r}$ where $(x)_r= x(x-1)\ldots (x-r+1)$ and $\beta>r$. We need to check that for $j=1,\ldots,r$,
\begin{gather*}
 \int_0^\infty  \left[x^{j} |\ell^{(j)}(x)| \right]^{(2\beta+e)/j} f(x) dx
 \lesssim   \int_0^\infty [ x^{b} + x^{j}]^{(2\beta+e)/j} x^{a-1} e^{- x^{b}} dx  \\
\lesssim  \int_0^\infty x^{b(2\beta+e)/j+a-1} e^{- x^{b}} dx  + \int_0^\infty x^{(2\beta+e)+a-1} e^{- x^{b}} dx= [z=x^b] \\
\lesssim   \int_0^\infty z^{(2\beta+e)/j+a/b-1} e^{- z} dz  + \int_0^\infty z^{(2\beta+e +a)/b-1} e^{- z} dz
\end{gather*}
which is finite since $b(2\beta+e)/j+a >0$ for $j=1,2,\ldots,r$.


Since  $b-r \in (0,1]$,
\begin{gather*}
|\ell^{(r)}(x+y) - \ell^{(r)}(x)| = (b)_r |(x+y)^{b-r} - x^{b-r}| 
 \leq (b)_r |y|^{b-r}
\end{gather*}
due to inequality $|z^A-w^A|\leq |z-w|^A$ for $A \in (0,1]$.  Here  $L_{\log}(x) = (b)_r $, $\beta=b$ and $\gamma=0$.

It is sufficient to check that
\begin{gather*}
 \int_0^\infty  \left[(x^{\beta}+x^{2\beta})\right]^{2} f(x) dx
 \lesssim    \int_0^\infty  (x^{2b+a-1}+x^{4b+a-1}) e^{- x^{b}} dx \\
\lesssim  \int_0^\infty [  z^{a/b+1}+ z^{a/b+3}] e^{- z} dz<\infty
 \end{gather*}
which holds.

For $a>1$, then we can take $\alpha=1$, and the corresponding  Weibull  density belongs to $\mathcal P( \beta, L(\cdot), \gamma, C_0, C_1, e, \Delta)$ due to the first part of Remark~\ref{rk:functionalclass}.

\subsubsection{Folded Student t distribution}

Now we take folded Student t distribution $f(x) = c_\nu (1+x^2)^{-(\nu+1)/2}$, $x>0$. Then $\alpha=1$ and $\ell(x) = -0.5(\nu+1) \log(1+x^2) $, and the
derivatives for large $x$ are 
\begin{gather*}
 \ell'(x) = - (\nu+1) x (1+x^2)^{-1}, \quad  \ell''(x) =  - (1+x^2)^{-1} + 2x^2(1+x^2)^{-2},\\
 \ell^{(2j)}(x) =  \sum_{s=0}^{j} c_{2j,s} \frac{x^{2s}}{ (1+x^2)^{j+s}},\quad
 \ell^{(2j+1)}(x) =  \sum_{s=0}^{j} c_{2j+1,s} \frac{x^{1+2s}} {(1+x^2)^{1+j+s}},
\end{gather*}
which is easy to prove by induction.
 Note that for any positive integer $k$, $| \ell^{(k)}(x)| \lesssim (1+x^2)^{-k/2}$.

Hence, for even derivatives,
\begin{gather*}
\int_0^\infty [x^{2j} |\ell^{(2j)}(x)|]^{(2\beta+e)/(2j)} f(x) dx  \lesssim
\sum_{s=0}^{j}
\int_0^\infty \left( \frac{x^2 }{(1+x^2)}\right)^{(1+s/j)(\beta+e/2) } f(x) dx
\end{gather*}
which is finite. Similarly, for odd derivatives,
\begin{gather*}
\int_0^\infty [x^{2j+1} |\ell^{(2j+1)}(x)|]^{(2\beta+e)/(2j+1)} f(x) dx  \\
\lesssim
\sum_{s=0}^{j}
\int_0^\infty \left( \frac{x^2 }{(1+x^2)}\right)^{(2\beta+e)(1+s+j)/(1+2j)} f(x) dx
<\infty.
\end{gather*}
Case $r=0$:
\begin{eqnarray*}
&|\ell(x+y) - \ell(x)| =  0.5(\nu+1) |\log\left(\frac{1+(x+y)^2}{1+x^2}\right)|\\
&\leq 0.5(\nu+1)A^{-1}\left[\frac{y (2x+y)}{1+x^2}\right]^A I(y>0) + 0.5(\nu+1)A^{-1}\left[\frac{|y| (2x-|y|)}{1+(x-|y|)^2}\right]^A I(y<0)\\
&\leq 0.5(\nu+1)A^{-1}\left[ \frac{2|y|(1+|y|)}{(1+x^2)^{1/2}} \right]^A I(y>0) + 0.5(\nu+1)A^{-1}\left[ 2  |y|(1+|y|)  \right]^A I(y<0)\\
&\lesssim |y|^A  (1+|y|^A)
\end{eqnarray*}
using inequality $\log(1+x)\leq x^A/A$ for any $x\geq 0$ and any $A>0$. Then, $\beta = A$ for $A\in (0,1]$, $\gamma=\beta=A$ and $L_{\ell}(x)= C$. 
Condition $\int_0^\infty x^{2\beta }(1+x^{2\gamma})L_{\ell}^2(x) f(x) dx<\infty$ holds if
\begin{gather*}
 \int_0^\infty x^{2A }(1+x^{2A }) (1+x^2)^{-(\nu+1)/2} dx
 \leq C+\int_1^\infty x^{4A} (1+x^2)^{-(\nu+1)/2} dx<\infty
\end{gather*}
i.e. if $\beta=A<\nu/4$ (here $r_0=0$).  

Now fix a positive integer $r$. Since
\begin{eqnarray*}
|\ell^{(r)}(x+y) - \ell^{(r)}(x)| &\leq& |y| \sup_{z \in \langle x, x+y\rangle} |\ell^{(r+1)}(z)| \lesssim |y| \sup_{z \in \langle x, x+y\rangle} (1+z^2)^{-(r+1)/2} \\
&\lesssim& |y| (1+x^2)^{-(r+1)/2} I(y>0) + |y| I(y<0) \lesssim |y|.
\end{eqnarray*}
Therefore, for any integer $r\geq 1$, the first condition is satisfied with $\beta=r+1$, $L_{\ell}(x)=C $ and $\gamma=0$.
 Condition $\int_0^\infty x^{2\beta }(1+x^{4r_0+2\gamma})L_{\ell}^2(x) f(x) dx<\infty$ holds if $\beta=r+1 < \nu/2$ and since $\beta=r+1 \geq 2$, we also need $2\beta+4r_0 <\nu$. Since $r_0=\lceil \beta/2\rceil-1<\beta/2$ and $\beta$ is an integer, we can write this condition as $\beta=r+1$ where $a_r<\nu$ where for even $r=2k$ $a_{r}  = 4r+2$ and for odd $r=2k+1$ $a_{r} = 4r$. For instance, $a_1=4$, $a_2 = 10$, $a_3=12$, $a_4 = 18$ etc.


Therefore, the conditions on $\beta$ and $L_\ell(x)$ given $\nu$ can be summarised as follows: $L_\ell(x)=C$  and
\begin{itemize}
\item  $\nu \in [1,4]$: $\beta<\nu/4$,  $\gamma=\beta$.


\item $\nu \in (a_{r},a_{r+1}]$: $\beta=r+1$,  $\gamma=0$.

\end{itemize}

\subsubsection{Frechet distribution}

Consider a Frechet-type distribution with density $f(x) = c_b x^{-b-1} e^{-x^{-b}}$, $x>0$, for some $b>0$. This density does not belong to a logarithmic  H\"older class. For simplicity we consider a bound of the type $|f(x)-f(x+y)|\leq L(x) |y|^\beta (1+|y|^\gamma)$ with $r=0$, i.e. with $\beta \leq 1$.
Hence, for $|y|\leq \Delta$, $x>0$ and $x>-y$,
\begin{gather}
| x^{-b-1} e^{-x^{-b}}-(x+y)^{-b-1} e^{-(x+y)^{-b}}| \leq |y| \sup_{z \in \langle x, x+y\rangle} [(b+1)z^{-b-2} + bz^{-2b-2}] e^{-z^{-b}}\notag\\
  \leq  (b+1) |y| \left[ \sup_{z \in \langle x, x+y\rangle}  z^{-b-2} e^{-z^{-b}} +  \sup_{z \in \langle x, x+y\rangle}  z^{-2b-2} e^{-z^{-b}}\right]. \label{eq:Frechet2Terms}
\end{gather}
For $a=b+2$ and for $a=2$, consider
 $\sup_{z \in \langle x, x+y\rangle}  z^{-b-a} e^{-z^{-b}}$.
Function $z^{-b-a} e^{-z^{-b}}$ achieves the maximum on the whole semiline at $x^\star_a = (1+a/b)^{-1/b}$. Hence, if $x^\star_a \in \langle x, x+y\rangle$ then the supremum is achieved at this point. If $\min(x,x+y) >  x^\star_a$ then the supremum is achieved at $\min(x,x+y)$, and if $\max(x,x+y)<x^\star_a$ then the supremum is achieved at $\max(x,x+y)$.

\begin{enumerate}
\item $\max(x,x+y) <  x^\star_a$. If $y\leq 0$ then the condition is $x <  x^\star_a$ and the supremum is $x^{-b-a} e^{-x^{-b}}$.
 If $y>0$, then the supremum is
\begin{gather*}
\sup_{z \in \langle x, x+y\rangle}  z^{-b-a} e^{-z^{-b}} = (x+y)^{-b-a} e^{-(x+y)^{-b}}\\ \leq x^{-b-a} e^{-x^{-b}} (1+  |y| ) \max(1, 2 b x^{-b-1})
\end{gather*}
using inequality
\begin{equation}\label{eq:IneqFrechet}
e^{x^{-b}-(x+y)^{-b} }\leq 1+ 2b |y| x^{-b-1}.
\end{equation}

We can unite the upper bound as $x^{-b-a} e^{-x^{-b}} (1+  |y| ) \max(1, 2 b x^{-b-1})$.

\item $\min(x,x+y) >  x^\star_a$. If $y\geq 0$ then the condition is $x >  x^\star_a$ and the supremum is $x^{-b-a} e^{-x^{-b}}$.
 If $y<0$ then
\begin{eqnarray*}
\sup_{z \in \langle x, x+y\rangle}  z^{-b-a} e^{-z^{-b}} = (x+y)^{-b-a} e^{-(x+y)^{-b}} \leq (x-\Delta)^{-b-a} e^{-x^{-b}}
\end{eqnarray*}
We can unite the upper bound as $(x-\Delta)^{-b-a} e^{-x^{-b}}$.

\item $\min(x,x+y) \leq  x^\star_a \leq \max(x,x+y)$, that is, $|x-x^\star_a|\leq |y|\leq \Delta$. Let's write the supremum as a function of $x$, $y$ and $\Delta$:
\begin{eqnarray*}
\sup_{z \in \langle x, x+y\rangle}  z^{-b-a} e^{-z^{-b}} &=& {x_a^\star}^{-b-a} e^{-{x_a^\star}^{-b}} \leq (x - \Delta)^{-b-a} e^{-{(x+|y|)}^{-b}}\\
&\leq& (x - \Delta)^{-b-a} e^{- x^{-b}}[1+   |y| ]\max(1, 2b x^{-b-1})
\end{eqnarray*}
using \eqref{eq:IneqFrechet}.

\end{enumerate}
To apply the bounds to the cases $a=2$ and $a=b+2$, note that $\max(x,x+y)<  x^\star_a$ holds if  $x <  x^\star_a-\Delta$; and it holds for both values of $a$ if $x <  x^\star_{b+2}-\Delta$ since $x^\star_a$ decreases in $a$.

Then, the upper bound in \eqref{eq:Frechet2Terms} can be written as
\begin{gather*}
| x^{-b-1} e^{-x^{-b}}-(x+y)^{-b-1} e^{-(x+y)^{-b}}| \leq  (b+1) |y|  e^{- x^{-b}}[1+   |y| ] \times \\
 \times \max(1, 2b x^{-b-1})\left[(x - \Delta)^{-b-a} I(x>x^\star_{b+2}-\Delta) + x^{-b-a} I(x < x^\star_{b+2}-\Delta) \right],
\end{gather*}
i.e. $\gamma=1$ and
\begin{eqnarray*}
L(x) &=& (b+1) e^{- x^{-b}} \max(1, 2b x^{-b-1})\\
&& \times \sum_{a \in \{2, b+2\}}\left[(x - \Delta)^{-b-a} I(x>x^\star_{b+2}-\Delta) + x^{-b-a} I(x < x^\star_{b+2}-\Delta) \right].
\end{eqnarray*}
Now we check the integrability condition:
\begin{eqnarray*}
&&\int_0^\infty \left(  x^{\beta}(1+x)  L(x) \right)^2 f(x) dx
\leq 2\int_0^\infty \left(  x^{\beta}(1+x)    \max(1, x^{-b-1}) \right)^2 \\
&& \times \sum_{a \in \{2, b+2\}}\left[(x - \Delta)^{-b-a} x^{b+1} I(x>x^\star_{b+2}-\Delta) + x^{-b-a} I(x < x^\star_{b+2}-\Delta) \right]^2 f(x) dx\\
&\leq& 8\sum_{a \in \{2, b+2\}}\int_{x^\star_{b+2}-\Delta}^\infty \left[  x^{\beta}    \max(x, x^{-b-1}) (x - \Delta)^{-b-a} x^{b+1}\right]^2 c_b x^{-b-1} e^{-x^{-b}} dx\\
& &+ 8\sum_{a \in \{2, b+2\}}\int_0^{x^\star_{b+2}-\Delta} \left[  x^{\beta}   \max(x, x^{-b-1})  x^{1-a}  \right]^2 c_b x^{-b-1} e^{-x^{-b}} dx\\
&\leq& 8\sum_{a \in \{2, b+2\}} \int_0^{x^\star_{b+2}-\Delta}    x^{2(\beta-b-a)-1}   c_b x^{-b-1} e^{-x^{-b}} dx\\
&&+8\sum_{a \in \{2, b+2\}}\int_{x^\star_{b+2}-\Delta}^{\infty} \left(  (x - \Delta)^{-b-a} x^{\beta+b+1} \right)^2 c_b x^{-b-1} e^{-x^{-b}} dx
\end{eqnarray*}
The first integral is finite. The second integral is finite if $ -2a+2\beta +4 <b$ for $a=2$ and for $a=b+2$, i.e. for   $ \beta   <b/2+1$ which holds for any $\beta \in(0,1]$.

\section*{Acknowledgements}
The authors are grateful to the Royal Society for financial support for mutual visits (International Exchange Grant IE140183), and to Oleg Lepski for fruitful discussions about the functional classes.


\end{document}